\documentclass[11pt]{amsart}
\usepackage{amsmath}
\usepackage{amssymb}
\usepackage{amsfonts}
\usepackage{amsthm}  
\usepackage{latexsym}
\usepackage{url} 
\usepackage{enumerate} 
\usepackage{mathdots}
\usepackage{dutchcal} 
\usepackage{verbatim} 
\usepackage{mathtools} 
\usepackage[pdftex]{hyperref}
\usepackage[nameinlink]{cleveref} 
\crefname{thm}{theorem}{theorems}
\crefname{lem}{lemma}{lemmas}
\crefname{cor}{corollary}{corollaries}
\crefname{prop}{proposition}{propositions}
\crefname{defn}{definition}{definitions}
\crefname{eg}{example}{examples}
\crefname{xca}{exercise}{exercises}
\crefname{conj}{conjecture}{conjectures}
\crefname{rmk}{remark}{remarks}
\crefname{qst}{question}{questions}
\crefname{obs}{observation}{observations}

\newtheorem{thm}{Theorem}[section]
\newtheorem*{thm*}{Theorem} 
\newtheorem{lem}[thm]{Lemma}  
\newtheorem{cor}[thm]{Corollary}
\newtheorem{prop}[thm]{Proposition}
\theoremstyle{definition}
\newtheorem{defn}[thm]{Definition}

\theoremstyle{remark}
\newtheorem{rmk}[thm]{Remark}

\newtheorem{obs}[thm]{Observation}
\newtheorem*{note}{Note} 
\numberwithin{equation}{section}

\newcommand{\ip}[2]{\langle #1 , #2 \rangle}    

\newcommand\diag{\operatorname{diag}}   
\newcommand\C{\mathbb C}    
\newcommand\Z{\mathbb Z}    
\newcommand\N{\mathbb N}

\newcommand\style{\mathfrak}
\newcommand\st{\operatorname{St}} 
\newcommand\GL{\mathbf{GL}} 
\newcommand\G{\mathbf{G}} 
\newcommand\Hbf{\mathbf{H}} 
\newcommand\X{\mathbf{X}} 
\newcommand\prm{\mathfrak p_F} 
\newcommand\of{\mathcal O_F} 
\newcommand\Exp{\mathcal{Exp}} 
\newcommand\wt{\widetilde} 

\DeclareMathOperator{\spn}{span} 
\DeclareMathOperator{\Int}{Int}  
\DeclareMathOperator{\Ad}{Ad}  
\DeclareMathOperator{\ind}{Ind} 
\DeclareMathOperator{\Hom}{Hom} 
\DeclareMathOperator{\Sym}{Sym} 
\DeclareMathOperator{\Gal}{Gal} 

\DeclarePairedDelimiter{\floor}{\lfloor}{\rfloor}

\begin{document}
\title{Relative discrete series representations for two quotients of $p$-adic $\mathbf{GL}_n$}
\author{Jerrod Manford Smith}
\address{University of Toronto, Department of Mathematics, Toronto, Canada}
\email{jerrod.smith@utoronto.ca}
\curraddr{University of Maine, Department of Mathematics \& Statistics, Orono, Maine}
\email{jerrod.smith@maine.edu}
\thanks{The author was partially supported by the Natural Sciences and Engineering Research Council, Canada Graduate Scholarship program and the Government of Ontario, Ontario Graduate Scholarship program.}
\subjclass[2010]{Primary 22E50; Secondary 22E35}
\keywords{$p$-adic symmetric space, relative discrete series, Casselman's Criterion}
\date{October 23, 2017}
\begin{abstract}
We provide an explicit construction of representations in the discrete spectrum of two $p$-adic symmetric spaces.  We consider $\mathbf{GL}_n(F) \times \mathbf{GL}_n(F) \backslash \mathbf{GL}_{2n}(F)$ and $\mathbf{GL}_n(F) \backslash \mathbf{GL}_n(E)$, where $E$ is a quadratic Galois extension of a nonarchimedean local field $F$ of characteristic zero and odd residual characteristic. The proof of the main result involves an application of a symmetric space version of Casselman's Criterion for square integrability due to Kato and Takano.
\end{abstract}
\maketitle
\section{Introduction}
Let $F$ be a nonarchimedean local field of characteristic zero and odd residual characteristic.   Let $G = \G(F)$ be the $F$-points of a connected reductive group defined over $F$.   Let $\theta$ be an $F$-involution (order two $F$-automorphism) of $\G$ and let $H = \G^\theta(F)$ be the group of $\theta$-fixed points in $G$.   The quotient $H \backslash G$ is a $p$-adic symmetric space.  An irreducible representation of $G$ that occurs in the discrete spectrum of $H\backslash G$ (an irreducible subrepresentation of $L^2(Z_GH\backslash G)$, where $Z_G$ is the centre of $G$) is called a relative discrete series (RDS) representation.  
In this paper, we construct an infinite family of RDS representations for $H\backslash G$, that do not appear in the discrete spectrum of $G$, in two cases:
\begin{enumerate}
\item The linear case: $G = \GL_{2n}(F)$ and $H = \GL_n(F) \times \GL_n(F)$.
\item The Galois case: $G = \GL_n(E)$ and $H = \GL_n(F)$, where $E/F$ is a quadratic extension.
\end{enumerate}
In a more general setting, Murnaghan has constructed relatively supercuspidal (\textit{cf.}~\Cref{defn-rsc}) representations that are not supercuspidal \cite{murnaghan2011a, murnaghan2016-pp}.  Her construction is also via parabolic induction from representations of $\theta$-elliptic Levi subgroups (\textit{cf.}~\Cref{theta-elliptic-defn}) and provided the initial motivation for this work.  We obtain a special case of Murnaghan's results in our setting (\textit{cf.}~\Cref{cor-sc-main-thm}); however, we apply completely different methods.

The study of harmonic analysis on $H\backslash G$ is of interest due to connections with non-vanishing of global period integrals, functoriality and poles of $L$-functions (see, for instance, \cite{jacquet--rallis1996, kable2004, feigon--lapid--offen2012}). For example, often $\GL_n(F) \times \GL_n(F)$-distinction of a representation of $\GL_{2n}(F)$ is equivalent to the existence of a nonzero Shalika model (\textit{cf.}~\Cref{rmk-lin-Shalika}).
In addition, $H\backslash G$ is a spherical variety and its study fits into the general framework of \cite{sakellaridis--venkatesh2017}.  In a broad sense, the work of Sakellaridis--Venkatesh lays the formalism and foundations of a relative Langlands program (\textit{cf.} \cite{prasad2015-pp}).  The aim of such a program is to fully understand the link between global automorphic period integrals and local harmonic analysis.  
For certain $p$-adic spherical varieties $X$, Sakellaridis--Venkatesh give an explicit Plancherel formula describing $L^2(X)$ up to a description of the discrete spectrum \cite[Theorem 6.2.1]{sakellaridis--venkatesh2017}.
In addition, the results of \cite{sakellaridis--venkatesh2017} include a description of the discrete spectrum of $p$-adic symmetric spaces in terms of toric families of RDS. In fact, they show that \cite[Conjecture 9.4.6]{sakellaridis--venkatesh2017} is true for strongly factorizable spherical varieties.
The decomposition of the norm on the discrete spectrum $L^2_{\operatorname{disc}}(X)$ provided by this result is not necessarily a direct integral, the images of certain intertwining operators packaged with the toric families of RDS may be non-orthogonal.
On the other hand, Sakellaridis and Venkatesh believe that their conjectures on $\X$-distinguished Arthur parameters may give a canonical choice of mutually orthogonal toric families of RDS that span $L^2_{\operatorname{disc}}(X)$ \cite[Conjectures 1.3.1 and 16.2.2]{sakellaridis--venkatesh2017}.
However, Sakellaridis--Venkatesh do not give an explicit description of the RDS required to build the toric families.  For this reason, an explicit construction of RDS for $p$-adic symmetric spaces is a step towards completing the picture of the discrete spectrum of $H\backslash G$.  Kato and Takano \cite{kato--takano2010} have shown that any $H$-distinguished discrete series representation of $G$ is a RDS.  Thus, we are interested in constructing RDS representations of $G$ that are in the complement of the discrete spectrum of $G$.

We state our main theorem below, after giving the necessary definitions. 
A $\theta$-stable Levi subgroup $L$ of $G$ is $\theta$-elliptic if it is not contained in any proper $\theta$-split parabolic subgroup, where a parabolic subgroup $P$ is $\theta$-split if $\theta(P)$ is opposite to $P$.  
An element $g\in G$ is said to be $\theta$-split if $\theta(g) = g^{-1}$ and a $\theta$-stable subset $Y$ of $G$ is $\theta$-split if every element $y\in Y$ is $\theta$-split.  
An $F$-torus $S$ is $(\theta,F)$-split if it is both $F$-split and $\theta$-split.
A representation $\tau$ of a Levi subgroup $L$ of $G$ is regular if for every non-trivial element $w\in N_G(L)/L$ we have that the twist ${}^w \tau = \tau ( w^{-1} (\cdot) w)$ is not equivalent to $\tau$.

Let $n\geq 2$ (respectively, $n\geq 4$) and let $G$ be equal to $\GL_{2n}(F)$ (respectively, $\GL_{n}(E)$) and let $H$ be equal to $\GL_{n}(F) \times \GL_n(F)$ (respectively, $\GL_n(F)$).
Let $A_0$ be a  $\theta$-stable maximal $F$-split torus of $G$ containing a fixed maximal $(\theta, F)$-split torus $S_0$. 
Let $L_0 = C_G( (A_0^\theta)^\circ)$ be a minimal $\theta$-elliptic Levi subgroup of $G$ containing $A_0$.  
We make a particular choice $\Delta^{ell}$ of simple roots for the root system $\Phi(G,A_0)$ (\textit{cf.}~$\S$\ref{sec-theta-ell-levi}). 
The $F$-split component of the centre of $L_0$ is determined by a proper nonempty subset $\Delta^{ell}_{\min}$ of $\Delta^{ell}$. 
A Levi subgroup $L$ is a standard-$\theta$-elliptic Levi subgroup if $L$ is standard with respect to $\Delta^{ell}$ and contains $L_0$.
The main result of this paper is the following.
\begin{thm*}[\Cref{main-RDS-thm}]
Let $\Omega^{ell} \subset {\Delta^{ell}}$ be a proper subset such that $\Omega^{ell}$ contains $\Delta^{ell}_{\min}$.
Let $Q = Q_{\Omega^{ell}}$ be the proper $\Delta^{ell}$-standard parabolic subgroup associated to the subset $\Omega^{ell}$.  The parabolic subgroup $Q$ is $\theta$-stable and has $\theta$-stable standard-$\theta$-elliptic Levi subgroup $L = L_{\Omega^{ell}}$.
Let $\tau$ be a regular $L^{\theta}$-distinguished discrete series representation of $L$.
The parabolically induced representation $\pi = \iota_Q^G\tau$ is an irreducible $H$-distinguished relative discrete series representation of $G$.  
Moreover, $\pi$ is in the complement of the discrete series of $G$.
\end{thm*}

\begin{rmk}
The representations constructed in \Cref{main-RDS-thm} are induced from discrete series, and are therefore tempered (generic) representations of $G$.  In particular, one observes that $L^2(Z_GH\backslash G)$ contains the $H$-distinguished discrete series representations of $G$ \cite{kato--takano2010}, as well as certain tempered representations that do not appear in $L^2(Z_G\backslash G)$.
\end{rmk}

Outside of two low-rank examples considered by Kato--Takano \cite[$\S$5.1-2]{kato--takano2010}, \Cref{main-RDS-thm} provides the first construction of a family of non-discrete relative discrete series representations.
In \Cref{cor-infinite-family}, we show that there are infinitely many equivalence classes of representations of the form constructed in \Cref{main-RDS-thm}.
However, we do not prove that our construction exhausts the discrete spectrum in these two cases (\textit{cf.}~\Cref{rmk-exhaustion}).  It appears that the major obstruction to showing that an irreducible $H$-distinguished representation of $G$ is not relatively discrete is establishing non-vanishing of the invariant forms $r_P\lambda$ defined by Lagier and Kato--Takano \cite{lagier2008,kato--takano2008} (\textit{cf.}~$\S$\ref{sec-r-P-lambda}). 

In the author's PhD thesis, a similar (more restricted) construction is carried out for the case $G = \GL_{2n}(E)$ and $H = \mathbf U_{E/F}$ a quasi-split unitary group \cite[Theorem 5.2.22]{smith-phd2017}.  It is work in progress to extend the construction to arbitrary symmetric quotients of the general linear group. Some modification will be required for this generalization; the representations constructed in \Cref{main-RDS-thm} are generic and no such representation can be distinguished by the symplectic group \cite{heumos--rallis1990}.  It is expected that the Speh representations form the discrete spectrum of $\mathbf{Sp}_{2n}(F) \backslash \GL_{2n}(F)$ \cite{offen--sayag2007, sakellaridis--venkatesh2017}.

We now give an outline of the content of the paper.  In  \Cref{sec-notation}, we establish notation and our conventions in the linear and Galois cases. In \Cref{sec-parabolic-inv}, we review basic results on tori and parabolic subgroups relevant to the study of harmonic analysis on $H \backslash G$. Here we introduce the notion of a $\theta$-elliptic Levi subgroup.  \Cref{sec-rep-theory}, contains a review of Kato--Takano's generalization of Casselman's Criterion, preliminaries on distinguished representations and some results on the exponents of induced representations.  The most important results in this section are \Cref{reduction-standard-split}, \Cref{red-to-ind-exp} and \Cref{non-dist-gen-eig-sp}.  In \Cref{sec-structure-lin-Gal}, we give explicit descriptions of the tori, parabolic subgroups, and simple roots needed for our work in the linear and Galois cases (\textit{cf}.~\Cref{max-theta-split,ind-subgp-prop}).  
The main result, \Cref{main-RDS-thm}, is stated and proved in \Cref{sec-main-theorem}; however, several preliminary results required for the proof are deferred until \Cref{sec-exp-dist-pi-N}.  In \Cref{sec-inducing-data}, we briefly survey the literature on distinguished discrete series representations in the linear and Galois cases. In addition, we establish the existence of infinite families of inducing representations in \Cref{lem-L-theta-dist}, from which we can deduce \Cref{cor-infinite-family}.  Finally, in \Cref{sec-exp-dist-pi-N}, we assemble the technical results, on the exponents and distinction of Jacquet modules, required to prove \Cref{main-RDS-thm}. The main results of the final section are \Cref{prop-reduce-Cas-ind,no-dist-unit-exp}.
 
\section{Notation and conventions}\label{sec-notation}
Let $F$ be a nonarchimedean local field of characteristic zero and odd residual characteristic. 
Let $\of$ be the ring of integers of $F$ with prime ideal $\prm$. 
Let $E$ be a quadratic Galois extension of $F$.  
Fix a generator $\varepsilon$ of the extension $E/F$ such that $E = F(\varepsilon)$.  
Let $\sigma \in \Gal(E/F)$ be a generator of the Galois group of $E$ over $F$. 

Let $\G$ be a connected reductive group defined over $F$ and let $G = \G(F)$ denote the group of $F$-points.
Let $e$ be the identity element of $G$.
We let $Z_G$ denote the centre of $G$ while $A_G$ denotes the $F$-split component of the centre of $G$.
As is the custom, we will often abuse notation and identify an algebraic group defined over $F$ with its group of $F$-points. 
When the distinction is to be made, we will use boldface to denote the algebraic group and regular typeface to denote the group of $F$-points.
For any $F$-torus $\mathbf{A}$ of $\G$, we let $A^1$ denote the group of $\of$-points $A^1=\mathbf A (\of)$.

Let $P$ be a parabolic subgroup of $G$ and let $N$ be the {unipotent radical} of $P$. 
The modular character of $P$ is given by $\delta_P(p) = | \det \Ad_{\style n}(p) |$, for all $p\in P$, where $\Ad_{\style n}$ denotes the adjoint action of $P$ on the Lie algebra $\style n$ of $N$ \cite{Casselman-book}.

Let $\theta$ be an $F$-involution of $\G$, that is, an order-two automorphism of $\G$ defined over $F$.
Define $\Hbf=\G^\theta$ to be the closed subgroup of $\theta$-fixed points of $\G$.
The quotient $H\backslash G$ is a $p$-adic symmetric space.

\begin{defn}\label{defn-involution-action}
We say that an involution $\theta_1$ of $G$ is $G$-conjugate (or $G$-equivalent) to another involution $\theta_2$ if there exists $g\in G$ such that $\theta_1 = \Int g^{-1} \circ \theta_2 \circ \Int g$, 
where $\Int g$ denotes the inner $F$-automorphism of $\G$ given by $\Int g(x) = g x g^{-1}$, for all $x\in \G$. 
We write $g\cdot \theta$ to denote the involution $\Int g^{-1} \circ \theta \circ \Int g$. 
\end{defn}

Let $\GL_n$ denote the general linear group of $n$ by $n$ invertible matrices. As is customary, we denote the block-upper triangular parabolic subgroup of $\GL_n$, corresponding to a partition $(\underline{m}) = (m_1,\ldots, m_k)$ of $n$, 
by $\mathbf P_{(\underline{m})}$, with block-diagonal Levi subgroup $\mathbf M_{(\underline{m})} \cong \prod_{i=1}^k \GL_{m_i}$ and unipotent radical $\mathbf N_{(\underline{m})}$.  
We use $\diag(a_1,a_2, \ldots, a_n)$ to denote an $n \times n$ diagonal matrix with entries $a_1,\ldots, a_n$.

For any $g,x \in G$, we write ${}^g x = gxg^{-1}$. For any subset $Y$ of $G$, we write ${}^g Y = \{ {}^g y : y \in Y\}$.
Let $C_G(Y)$ denote the centralizer of $Y$ in $G$ and let $N_G(Y)$ be the normalizer of $Y$ in $G$.
Given a real number $r$ we let $\floor{r}$ denote the greatest integer that is less than or equal to $r$. 
We use $\widehat {(\cdot)}$ to denote that a symbol is omitted.  For instance, $\diag(\widehat{a_1}, a_2, \ldots, a_n)$ may be used to denote the diagonal matrix $\diag(a_2,\ldots, a_n)$.
\subsection{The linear case}\label{sec-notation-lin}
In the linear case, we set $G = \GL_n(F)$, where $n\geq 4$ is an even integer.
Let $\theta$ denote the inner involution of $G$ given by conjugation by the matrix
\begin{align*}
w_\ell = \left( \begin{matrix} & & 1 \\ & \iddots & \\ 1 & & \end{matrix} \right),
\end{align*}
that is, for any $g\in \G$, we have 
\begin{align*}
\theta(g) &= \Int w_\ell (g) = w_\ell g w_\ell^{-1}.
\end{align*} 
The element $w_\ell$ is diagonalizable over $F$; in particular, there exists $x_\ell \in \GL_n(F)$ such that 
\begin{align}\label{eq-w-ell-diag}
x_\ell w_\ell x_\ell^{-1} = \diag(1_{n/2}, -1_{n/2}),
\end{align}
where $1_{n/2}$ denotes the ${n/2}\times {n/2}$ identity matrix.
It follows that $H = x_\ell^{-1} M_{(n/2,n/2)} x_\ell$, where $M_{(n/2,n/2)}$ is the standard Levi subgroup of $G$ of type $(n/2,n/2)$.  Thus, in the linear case, $H = \G^\theta(F)$ is isomorphic to $\GL_{n/2}(F)\times \GL_{n/2}(F)$.
\subsection{The Galois case}\label{sec-notation-Gal}
In the Galois case, for $n\geq 4$, we let $\G = R_{E/F}\GL_n$ be the restriction of scalars of $\GL_n$ with respect to $E/F$.  We identify the group $G$ of $F$-points with $\GL_n(E)$.
The non-trivial element $\sigma$ of the Galois group of $E$ over $F$ gives rise to an $F$-involution $\theta$ of $G$ given by coordinate-wise Galois conjugation
\begin{align*}
\theta ( (a_{ij}) ) = (\sigma(a_{ij})),
\end{align*}
where $(a_{ij}) \in G$. 
In the Galois case,  we have that $H = \G^\theta(F)$ is equal to $\GL_n(F)$.   
\subsection{Choices of particular group elements and supplementary involutions}\label{sec-choice-elts}
For a positive integer $r$, we'll write $\G_r$ for $\GL_r$ with $F$-points $G_r$ in the linear case, and similarly for $R_{E/F} \GL_r$ with $F$-points $G_r \simeq \GL_r(E)$, in the Galois case. 
Write $J_r$ for the $r\times r$-matrix in $G_r$ with unit anti-diagonal
\begin{align*}
J_r = \left( \begin{matrix} & & 1 \\ & \iddots & \\ 1 & & \end{matrix} \right).
\end{align*}
Note that $w_\ell=J_n$. 
In the linear case, $\theta_r$ will denote the inner involution $\Int J_r$ of $\G_r$ with fixed points $\Hbf_r$.
In the Galois case, we let $\theta_r$ denote the $F$-involution of $\G_r$ given by coordinate-wise Galois conjugation; then $H_r = \GL_r(F)$ is the group of $F$-points of the $\theta_r$-fixed subgroup of $\G_r$.
In the Galois case, for any positive integer $r$, there exists $\gamma_r \in G_r$ such that $\gamma_r^{-1} \theta_r(\gamma_r) = J_r  \in H_r$. 
For instance, if $r$ is even, then we may take
\begin{align*}
\gamma_r = \left( \begin{matrix} 
1 & & & &  &1 \\
   & \ddots & & & \iddots & \\
& & 1 & 1 & & \\
& & -\varepsilon & \varepsilon & & \\
& \iddots & & & \ddots & \\
-\varepsilon & & & & & \varepsilon
\end{matrix} \right),
\end{align*}
 where $E = F(\varepsilon)$, and if $r$ is odd, then we set
\begin{align*}
\gamma_r = \left( \begin{matrix} 
1 & & && &  &1 \\
   & \ddots & && & \iddots & \\
& & 1 &0& 1 & & \\
& & 0& 1 & 0& & \\
& & -\varepsilon &0& \varepsilon & & \\
& \iddots & && & \ddots & \\
-\varepsilon & & & & & & \varepsilon
\end{matrix} \right).
\end{align*}
Define $\gamma = \gamma_n \in G$ and note that $w_\ell = J_n = \gamma^{-1}\theta(\gamma)$ is an order-two element of $H$.

In the Galois case, we define a second involution $\vartheta$ of $\G$, that is $G$-conjugate to $\theta$, by declaring that $\vartheta = \gamma \cdot \theta$ (\textit{cf.}~\Cref{defn-involution-action}). Explicitly,
 \begin{align}\label{eq-vartheta-defn}
 \vartheta  (g) = \gamma^{-1} \theta(\gamma g \gamma^{-1}) \gamma,
 \end{align}
 for any $g\in \G$.
 Since $w_\ell=\gamma^{-1}\theta(\gamma)$ is $\theta$-fixed, we have that 
 \begin{align}\label{Galois-involution-relation}
 \vartheta = \Int w_\ell \circ \theta = \theta \circ \Int w_\ell.
 \end{align}
 Similarly, for any positive integer $r$, we define 
 \begin{align}\label{vartheta-r-defn}
 \vartheta_r = \gamma_r \cdot \theta_r = \Int J_r \circ \theta_r = \theta_r \circ \Int J_r.
 \end{align}
 
In both cases, define $w_+ \in \GL_n(F) \subset \GL_n(E)$ to be the permutation matrix corresponding to the permutation of $\{1,\ldots, n\}$ given by
\begin{align*}
&\left\{\begin{array}{lll}
2i-1 & \mapsto i & 1 \leq i \leq \floor{n/2}+1 \\
2i & \mapsto n+1-i & 1 \leq i \leq \floor{n/2} 
\end{array}
\right. & (n \ \text{odd}),
\end{align*}
when $n$ is odd, and when $n$ is even by
\begin{align*}
&\left\{\begin{array}{lll}
2i-1 & \mapsto i & 1 \leq i \leq n/2 \\
2i & \mapsto n+1-i & 1 \leq i \leq n/2 
\end{array}
\right. &(n \ \text{even}).
\end{align*}
Remember that in the linear case we'll always assume that $n$ is even.
Finally, define 
\begin{align}\label{w-zero}
w_0 &= \left\{\begin{array}{lll}
w_+ & \text{in the linear case:} \ G = \GL_n(F), n\geq 4 \ \text{even}, \\
 {}^\gamma w_+= \gamma w_+ \gamma^{-1}  & \text{in the Galois case:} \ G = \GL_n(E), \ \text{any} \ n\geq 4.
\end{array}
\right. 
\end{align}  
\section{Symmetric spaces and associated parabolic subgroups}\label{sec-parabolic-inv}
For now, we work in general and let $G$ be an arbitrary connected reductive group over $F$, with $\theta$ and $H$ as in \Cref{sec-notation}.
An element $g\in G$ is said to be $\theta$-split if $\theta(g) = g^{-1}$. 
A subtorus $S$ of $G$ is $\theta$-split if every element of $S$ is $\theta$-split.
\subsection{Tori and root systems relative to involutions}\label{sec-tori-involution}
An $F$-torus $S$ contained in $G$ is $(\theta,F)$-split if $S$ is both $F$-split and $\theta$-split.
Let $S_0$ be a maximal $(\theta,F)$-split torus of $G$.
By \cite[Lemma 4.5(\rm{iii})]{helminck--wang1993}, there exists a $\theta$-stable maximal $F$-split torus $A_0$ of $G$ that contains $S_0$.
Let $\Phi_0 = \Phi(G,A_0)$ be the root system of $G$ with respect to $A_0$.  Let $W_0$ be the Weyl group of $G$ with respect to $A_0$.  Since $A_0$ is $\theta$-stable, there is an action of $\theta$ on the $F$-rational characters $X^*(A_0)$ of $A_0$. 
Explicitly, given $\chi \in X^*(A_0)$, we have
\begin{align*}
(\theta \chi) (a) = \chi (\theta(a)),
\end{align*}
for all $a \in A_0$.
Moreover, $\Phi_0$ is stable under the action of $\theta$ on $X^*(A_0)$.
Let $\Phi_0^\theta$ denote the subset of $\theta$-fixed roots in $\Phi_0$.

\begin{defn}\label{defn-theta-base}
A base $\Delta_0$ of $\Phi_0$ is called a $\theta$-base if for every positive root $\alpha \in \Phi_0^+$ with respect to $\Delta_0$ that is not fixed by $\theta$, we have that $\theta(\alpha) \in \Phi_0^-$.   
\end{defn}
As shown in \cite{helminck1988}, a $\theta$-base of $\Phi_0$ exists.  Let $\Delta_0$ be a $\theta$-base of $\Phi_0$.
Let $p: X^*(A_0) \rightarrow X^*(S_0)$ be the surjective homomorphism defined by restricting the $F$-rational characters on $A_0$ to the subtorus $S_0$.  The kernel of the map $p$ is the submodule $X^*(A_0)^\theta$ of $X^*(A_0)$ consisting of $\theta$-fixed $F$-rational characters.
The restricted root system of $H \backslash G$ (relative to our choice of $(A_0,S_0,\Delta_0)$) is defined to be
\begin{align*}
\overline \Phi_0 = p(\Phi_0)\setminus \{0\} = p(\Phi_0 \setminus \Phi_0^\theta).
\end{align*}
The set $\overline \Phi_0$ coincides with the set $\Phi(G,S_0)$ of roots with respect to $S_0$  and this is a (not necessarily reduced) root system by \cite[Proposition 5.9]{helminck--wang1993}.  
The set
\begin{align*}
\overline \Delta_0 = p(\Delta_0) \setminus \{0\} = p(\Delta_0 \setminus \Delta_0^\theta)
\end{align*}
is a base for the restricted root system $\overline \Phi_0$. 
Indeed, the linear independence of $\overline \Delta_0$ follows from the fact that $\Delta_0$ is a $\theta$-base and that $\ker p = X^*(A_0)^\theta$.
Given a subset $\overline \Theta \subset \overline \Delta_0$, define the subset 
\begin{align*}
[\overline\Theta ] = p^{-1}(\overline\Theta ) \cup \Delta_0^\theta
\end{align*}
of $\Delta_0$. 
Subsets of $\Delta_0$ of the form $[\overline\Theta ]$, for $\overline \Theta \subset \overline \Delta_0$, are said to be $\theta$-split.  
The maximal $\theta$-split subsets of $\Delta_0$ are of the form
$[\overline\Delta_0 \setminus\{\bar\alpha\}]$,  
where $\bar\alpha \in \overline\Delta_0$. 
\subsection{Parabolic subgroups relative to involutions}\label{sec-pblc-rel-inv}
Given a subset $\Theta$ of $\Delta_0$, one may canonically associate a $\Delta_0$-standard parabolic subgroup $P_\Theta$ of $G$ and a standard choice of Levi subgroup. 
Let $\Phi_\Theta $ be the subsystem of $\Phi_0$ generated by the simple roots $\Theta$. 
Let $\Phi_\Theta ^+$ be the set of positive roots determined by $\Theta$.
The unipotent radical $N_\Theta $ of $P_\Theta $ is generated by the root groups $N_\alpha$, where $\alpha \in \Phi_0^+\setminus \Phi_\Theta ^+$.
The parabolic subgroup $P_\Theta $ admits a Levi factorization $P_\Theta  = M_\Theta  N_\Theta $, where $M_\Theta $ is the centralizer in $G$ of the $F$-split torus
\begin{align*}
A_\Theta  = \left( \bigcap_{\alpha \in \Theta } \ker \alpha \right)^\circ.
\end{align*}
Here $(\cdot)^\circ$ indicates the Zariski-connected component of the identity.
In fact, $A_\Theta $ is the $F$-split component of the centre of $M_\Theta $ and $\Phi_\Theta $ is the root system $\Phi(M_\Theta, A_0)$ of $A_0$ in $M_\Theta $.

\begin{rmk}
When considering standard parabolic subgroups $P_\Theta $, associated to $\Theta \subset \Delta_0$, we will always work with the Levi factorization $P_\Theta  = M_\Theta N_\Theta $, where $M_\Theta  = C_G(A_\Theta )$ is the standard Levi subgroup of $P_\Theta $. 
\end{rmk}

Let  $M$ be any Levi subgroup of $G$.  
The $(\theta,F)$-split component of $M$ is the largest $(\theta,F)$-split torus $S_M$ contained in the centre of $M$. 
In fact, we have that $S_M$ is the connected component (of the identity) of the subgroup of $\theta$-split elements in the $F$-split component $A_M$, that is,
\begin{align}\label{eq-S-M-defn}
S_M = \left( \{ x \in A_M : \theta(x) = x^{-1}\} \right)^\circ.
\end{align}  

A parabolic subgroup $P$ of $G$ is called $\theta$-split if $\theta(P)$ is opposite to $P$.  In this case, $M = P\cap \theta(P)$ is a $\theta$-stable Levi subgroup of both $P$ and $\theta(P)=P^{op}$.
Given a $\theta$-split subset $\Theta \subset \Delta_0$, the $\Delta_0$-standard parabolic subgroup $P_\Theta  = M_\Theta N_\Theta$ is $\theta$-split. 
Any $\Delta_0$-standard $\theta$-split parabolic subgroup arises from a $\theta$-split subset of $\Delta_0$ \cite[Lemma 2.5(1)]{kato--takano2008}.
Let $S_\Theta $ denote the $(\theta,F)$-split component of $M_\Theta $. 
We have that
\begin{align}\label{eq-std-theta-split-torus}
S_\Theta  = \left( \{ s \in A_\Theta  : \theta(s) = s^{-1} \} \right)^\circ =  \left( \bigcap_{ \bar\alpha \in p(\Theta)} \ker (\bar\alpha: S_0 \rightarrow F^\times )  \right)^\circ.
\end{align}
For the second equality in \eqref{eq-std-theta-split-torus}, see \cite[$\S1.5$]{kato--takano2010}.  
For any $0 < \epsilon \leq 1$, define 
\begin{align}\label{eq-split-dominant-part}
S_\Theta ^-(\epsilon) = \{ s \in S_\Theta  : |\alpha(s)|_F \leq \epsilon, \ \text{for all} \ \alpha \in \Delta_0 \setminus \Theta \}.
\end{align}
We write $S_\Theta ^-$ for $S_\Theta ^-(1)$ and refer to $S_\Theta ^-$ as the dominant part of $S_\Theta$.

By \cite[Theorem 2.9]{helminck--helminck1998}, the subset $\Delta_0^\theta$ of $\theta$-fixed roots in $\Delta_0$ determines the $\Delta_0$-standard minimal $\theta$-split parabolic subgroup $P_0 = P_{\Delta_0^\theta}$.  By \cite[Proposition 4.7(\rm{iv})]{helminck--wang1993}, the minimal $\theta$-split parabolic subgroup $P_0$ has standard $\theta$-stable Levi $M_0 = C_G(S_0)$.   
Let $P_0 = M_0 N_0$ be the standard Levi factorization of $P_0$.

\begin{lem}\label{levi-theta-split}
Let $P$ be a $\theta$-split parabolic subgroup with $\theta$-stable Levi $M = P \cap \theta(P)$. The Levi subgroup $M$ is equal to the centralizer in $G$ of its $(\theta,F)$-split component $S_M$.
\end{lem}

\begin{proof}
The lemma follows immediately from \cite[Lemma 4.6]{helminck--wang1993}.
\end{proof}

\begin{lem}\label{levi-theta-split-2}
If $M$ is the centralizer in $G$ of a non-central $(\theta,F)$-split torus $S$, then $M$ is the Levi subgroup of a proper $\theta$-split parabolic subgroup of $G$.
\end{lem}

\begin{proof}
Let $S_0$ be a maximal $(\theta,F)$-split torus of $G$ containing $S$ and $A_0$ a $\theta$-stable maximal $F$-split torus of $G$ containing $S_0$.
The subgroup $M = C_G(S)$ is a $\theta$-stable Levi subgroup of $G$ since $S$ is a $\theta$-stable $F$-split torus.
Since $S$ is not central in $G$, $M$ is a proper Levi subgroup.
Moreover, since $S$ is contained in $S_0$, we have that $M_0$ is contained in $M$.
Let $P = MN_0$.  Note that $P$ is a closed subgroup containing $P_0$; therefore, $P$ is a proper parabolic subgroup of $G$ with Levi subgroup $M$.
It remains to show that $P$ is $\theta$-split.  Since $P_0$ is $\theta$-split, we have that $\theta(N_0) = N_0^{op}$ is the opposite unipotent radical of $N_0$.  
Since $M$ is $\theta$-stable, it follows that $\theta(P) = MN_0^{op}$ and this is the parabolic opposite to $P$.
\end{proof}

The minimal $\theta$-split parabolic subgroups of $G$ are not always $H$-conjugate \cite[Example 4.12]{helminck--wang1993}. On the other hand, the following result holds.
\begin{lem}[{\cite[Lemma 2.5]{kato--takano2008}}]\label{KT08-lem-2.5}
Let $S_0 \subset A_0$, $\Delta_0$, and $P_0 = M_0N_0$ be as above.
\begin{enumerate}
\item Any $\theta$-split parabolic subgroup $P$ of $G$ is conjugate to a $\Delta_0$-standard $\theta$-split parabolic subgroup by an element $g \in (\Hbf \mathbf M_0)(F)$.
\item If the group of $F$-points of the product $(\Hbf \mathbf M_0)(F)$ is equal to $HM_0$, then any $\theta$-split parabolic subgroup of $G$ is $H$-conjugate to a $\Delta_0$-standard $\theta$-split parabolic subgroup.
\end{enumerate}
\end{lem} 

Let $P = MN$ be a $\theta$-split parabolic subgroup and choose $g\in (\Hbf \mathbf M_0)(F)$ such that $P = gP_\Theta  g^{-1}$ for some $\Delta_0$-standard $\theta$-split parabolic subgroup $P_\Theta $.  
Since $g\in (\Hbf \mathbf M_0)(F)$ we have that $g^{-1}\theta(g) \in \mathbf M_0(F)$. 
Thus, we may take $S_M = g S_\Theta g^{-1}$.
For a given $\epsilon >0$, one may extend the definition \eqref{eq-split-dominant-part} of $S_\Theta ^-$ to the non-$\Delta_0$-standard torus $S_M$. 
Indeed, we may set
$S_M^-(\epsilon) = g S_\Theta ^-(\epsilon) g^{-1}$ 
and we define $S_M^- = S_M^-(1)$ with $S_M^1 = S_M(\of)$, as above.  

We give the following definition, following the terminology of Murnaghan \cite{murnaghan2016-pp}, in analogy with the notion of an elliptic Levi subgroup.
\begin{defn}\label{theta-elliptic-defn}
A $\theta$-stable Levi subgroup $L$ of $G$ is $\theta$-elliptic if and only if $L$ is not contained in any proper $\theta$-split parabolic subgroup of $G$.
\end{defn}

We note the following simple lemma, which follows immediately from \Cref{theta-elliptic-defn}. 
\begin{lem}\label{contain-theta-elliptic}
If a $\theta$-stable Levi subgroup $L$ of $G$ contains a $\theta$-elliptic Levi subgroup, then $L$ is $\theta$-elliptic.
\end{lem}

The following characterization of the $\theta$-elliptic property is also useful.

\begin{lem}\label{theta-elliptic-centre}
A $\theta$-stable Levi subgroup $L$ is $\theta$-elliptic if and only if $S_L = S_G$.
\end{lem}

\begin{proof}
If $L=G$, then there is nothing to do.  
Without loss of generality, $L$ is a proper subgroup of $G$.
Suppose that $L$ is $\theta$-elliptic.  
We have that $A_G$ is contained in $A_L$ and it follows that $S_G$ is contained in $S_L$.
If $S_L$ properly contains $S_G$, then $L$ is contained in the Levi subgroup $M = C_G(S_L)$.
By Lemma \ref{levi-theta-split-2}, $M$ is a Levi subgroup of a proper $\theta$-split parabolic subgroup.
It follows that $L \subset M$ is contained in a proper $\theta$-split parabolic subgroup. This contradicts the fact that $L$ is $\theta$-elliptic, so we must have that $S_L = S_G$.

On the other hand, suppose that $S_L$ is equal to $S_G$.  
Argue by contradiction, and suppose that $L$ is contained in a proper $\theta$-split parabolic $P=MN$ with $\theta$-stable Levi subgroup $M=P\cap\theta(P)$.
We have that $L = \theta(L)$ is contained in $M$ and $S_L \subset S_M$. By Lemma \ref{levi-theta-split}, $M$ is the centralizer of its $(\theta,F)$-split component $S_M$.
Since $M$ is a proper Levi subgroup of $G$, we have that $S_M$ properly contains $S_G$. 
However, by assumption $S_L=S_G$ is the largest $(\theta,F)$-split torus of contained $L$ and this is impossible.  
We conclude that $L$ must be $\theta$-elliptic.
\end{proof}

The next proposition appears in \cite{murnaghan2016-pp}.

\begin{prop}\label{theta-elliptic-theta-stable}
Let $Q$ be a parabolic subgroup of $G$. If $Q$ admits a $\theta$-elliptic Levi factor $L$, then $Q$ is $\theta$-stable.
\end{prop}

\begin{proof}
The subgroup $L$ is $\theta$-stable by definition.
For any root $\alpha$ of $A_L$ in $G$, one can show that $\theta \alpha = \alpha$.  
It follows that the unipotent radical of $Q$, hence $Q$, must be $\theta$-stable.
\end{proof}

\section{Distinguished representations and the Relative Casselman's Criterion}\label{sec-rep-theory}
\subsection{Distinguished representations}
All of the representations that we consider are on complex vector spaces.  
A representation $(\pi,V)$ of $G$ is smooth if for every $v\in V$ the stabilizer of $v$ in $G$ is an open subgroup.  A smooth representation $(\pi,V)$  of $G$ is admissible if, for every compact open subgroup $K$ of $G$, the subspace of $K$-invariant vectors $V^K$ is finite dimensional.  A smooth one-dimensional representation of $G$ is a quasi-character of $G$.  A character of $G$ is a unitary quasi-character.  Let $(\pi,V)$ be a smooth representation of $G$. If $\omega$ is a quasi-character of $Z_G$, then $(\pi,V)$ is an $\omega$-representation if $\pi$ has central character $\omega$.  

Let $\chi$ be a quasi-character of $H$.  We also let $\pi$ denote its restriction to $H$. 
\begin{defn}\label{defn-dist}
The representation $\pi$ is $(H,\chi)$-distinguished if the space $\Hom_H(\pi,\chi)$ is nonzero.  
If $\chi$ is the trivial character of $H$, then we refer to $(H,1)$-distinguished representations simply as $H$-distinguished. 
\end{defn}
Of course, in \Cref{defn-dist}, the subgroup $H = G^\theta$ may be replaced by any closed subgroup of $G$; however, we're only concerned with the symmetric subgroup setting.
As a first observation, we record the following lemma.
\begin{lem}\label{sub-quotient2}
Let $(\pi,V)$ be a finitely generated admissible representation of $G$. If $(\pi,V)$ is $H$-distinguished, then there exists an (irreducible) $H$-distinguished sub-quotient of $(\pi,V)$.
\end{lem}

The next lemma shows that distinction, relative to an involution $\theta$, depends only on the equivalence class of $\theta$ under the right action of $G$ on the set of involutions (\textit{cf.}~\Cref{defn-involution-action}).

\begin{lem}\label{orbit-dist}
The subgroup $G^{g\cdot \theta}$, of $g\cdot \theta$-fixed points in $G$, is $G$-conjugate to $G^\theta$. Precisely, we have $G^{g\cdot \theta} = g^{-1} (G^\theta) g$.
Moreover, a smooth representation $(\pi, V)$ of $G$ is $G^\theta$-distinguished if and only if $\pi$ is $G^{g\cdot \theta}$-distinguished.
\end{lem}

\begin{proof}
Let $h \in G^\theta$, then we have that
\begin{align*}
g\cdot \theta (g^{-1} h g) & = g^{-1} \theta( gg^{-1} h g g^{-1}) g
 = g^{-1} \theta(  h ) g
 = g^{-1}  h g 
\end{align*}
so $g^{-1}  h g$ is $g\cdot\theta$-fixed.  It follows that $g^{-1} (G^\theta) g \subset G^{g\cdot \theta}$.
Since conjugation by $g$ is an automorphism of $G$, it follows that $G^{g\cdot \theta} = g^{-1}( G^\theta )g$.

Let $\lambda$ be a nonzero element of $\Hom_{G^\theta}(\pi,1)$.  Define $\lambda ' = \lambda \circ \pi(g)$ and observe that $\lambda'$ is a nonzero $G^{g\cdot \theta}$-invariant linear functional on $V$. 
It follows that the map $\lambda \mapsto \lambda \circ \pi(g)$ is a bijection from  $\Hom_{G^\theta}(\pi,1)$ to  $\Hom_{G^{g\cdot \theta}}(\pi,1)$ with inverse $\lambda ' \mapsto \lambda' \circ \pi(g^{-1})$.
In particular, $\pi$ is ${G^\theta}$-distinguished if and only if $\pi$ is $G^{g\cdot \theta}$-distinguished.
\end{proof}
\subsection{Relative matrix coefficients}\label{sec-matrix-coeff}
Let $(\pi,V)$ be a smooth $H$-distinguished representation of $G$.  
Let $\lambda \in \Hom_H(\pi,1)$ be a nonzero $H$-invariant linear form on $V$ and let $v$ be a nonzero vector in $V$.
In analogy with the usual matrix coefficients, define a complex-valued function $\varphi_{\lambda,v}$ on $G$ by
$\varphi_{\lambda,v}(g) = \ip{\lambda}{\pi(g)v}$.  
We refer to the functions $\varphi_{\lambda,v}$ as relative matrix coefficients (with respect to $\lambda$) or as $\lambda$-relative matrix coefficients.
Since $\pi$ is a smooth representation, the relative matrix coefficient $\varphi_{\lambda,v}$ lies in $C^\infty(G)$, for every $v\in V$.
In addition, since $\lambda$ is $H$-invariant, the functions $\varphi_{\lambda,v}$ descend to well-defined functions on the quotient $H\backslash G$.
In analogy with the classical case, one makes the following definitions.
\begin{defn}\label{defn-rsc}
The representation $(\pi,V)$ is said to be
\begin{enumerate}
\item $(H,\lambda)$-relatively supercuspidal, or relatively supercuspidal with respect to $\lambda$, if and only if all of the $\lambda$-relative matrix coefficients are compactly supported modulo $Z_GH$.  
\item $H$-relatively supercuspidal if and only if $\pi$ is $(H,\lambda)$-relatively supercuspidal for every $\lambda \in \Hom_H(\pi,1)$.
\end{enumerate}
\end{defn}
Let $\omega$ be a unitary character of $Z_G$ and further suppose that $\pi$ is an $\omega$-representation.
\begin{defn}
The representation $(\pi,V)$ is said to be
\begin{enumerate}
\item $(H,\lambda)$-relatively square integrable, or relatively square integrable with respect to $\lambda$, if and only if all of the $\lambda$-relative matrix coefficients are square integrable modulo $Z_GH$.  
\item $H$-relatively square integrable if and only if $\pi$ is $(H,\lambda)$-relatively square integrable for every $\lambda \in \Hom_H(\pi,1)$.
\end{enumerate}
\end{defn}
Notice that we must also take the quotient of $G$ by the (noncompact) centre $Z_G$ in order to make sense of compactly supported (respectively, square integrable) functions on $H \backslash G$. Moreover, to integrate relative matrix coefficients over $Z_GH\backslash G$ we need a $G$-invariant measure on the quotient $Z_GH \backslash G$.  The centre $Z_G$ of $G$ is unimodular since it is abelian.  
The fixed point subgroup $H$ is also reductive (\textit{cf.}~\cite[Theorem 1.8]{digne--michel1994}) and thus unimodular. 
It follows that there exists a $G$-invariant measure on the quotient $Z_GH \backslash G$ by \cite[Proposition 12.8]{Robert-book}.

\begin{note}
When $H$ is understood, we refer to $H$-relatively supercuspidal (respectively, $H$-relatively square integrable) representations simply as relatively supercuspidal (respectively, relatively square integrable).
\end{note} 

\begin{defn}
If $(\pi,V)$ is an irreducible subrepresentation of $L^2(Z_GH\backslash G)$, then we say that $(\pi,V)$ occurs in the discrete spectrum of $H\backslash G$.  
In this case, we say that $(\pi,V)$ is a relative discrete series (RDS) representation.
\end{defn}

\subsection{Parabolic induction and Jacquet restriction}
Let $P$ be a parabolic subgroup of $G$ with Levi subgroup $M$ and unipotent radical $N$.  
 Given a smooth representation $(\rho, V_\rho)$ of $M$ we may inflate $\rho$ to a representation of $P$, also denoted $\rho$, by declaring that $N$ acts trivially.
 We define the representation $\iota_P^G \rho$ of $G$ to be the induced representation $\ind_P^G (\delta_P^{1/2} \otimes \rho)$.
We refer to the functor $\rho \mapsto \iota_P^G\rho$ as (normalized) parabolic induction.  When it is more convenient (\textit{cf.}~\Cref{obs-isom-type-RDS}, $\S$\ref{sec-inducing-data}-\ref{sec-exp-dist-pi-N}), we also use the Bernstein--Zelevinsky notation for parabolic induction on general linear groups \cite{bernstein--zelevinsky1977, zelevinsky1980}.

Let $(\pi,V)$ be a smooth representation of $G$.  Let $(\pi_N, V_N)$ denote the Jacquet module of $\pi$ along $P$, normalized by $\delta_P^{-1/2}$.
Explicitly, if $V(N) = \spn\{ \pi(n)v - v : n\in N , v\in V\}$, then $V_N = V / V(N)$ and $\pi_N(p)(v+V(N)) = \delta_P^{-1/2}(p) \pi(p) v + V(N)$,
for all $p\in P$, $v+V(N) \in V_N$.
Since $\delta_P$ is trivial on $N$, we see that $(\pi_N,V_N)$ is a representation of $P$ on which $N$ acts trivially.
We will regard $(\pi_N, V_N)$ as a representation of the Levi factor $M \cong P/ N$ of $P$.

We will now give a statement of the Geometric Lemma \cite[Lemma 2.12]{bernstein--zelevinsky1977}, which is a fundamental tool in our work and the study of induced representations in general.  First, we recall two results on double-coset representatives.

\begin{lem}[{\cite[Proposition 1.3.1]{Casselman-book}}]\label{lem-nice-reps}
Let $\Theta$ and $\Omega$ be subsets of $\Delta_0$.
The set 
\begin{align*}
[W_\Theta \backslash W_0 /W_\Omega] = \{ w \in W_0 : w\Omega, w^{-1}\Theta \subset \Phi^+ \}
\end{align*}
provides a choice of Weyl group representatives for the double-coset space $P_\Theta \backslash G / P_\Omega$.
\end{lem}

\begin{prop}[{\cite[Proposition 1.3.3]{Casselman-book}}]\label{casselman1-3-3}
Let $\Theta,\Omega \subset \Delta_0$ and let $w \in [W_\Theta \backslash W_0 /W_\Omega]$.
\begin{enumerate}
\item The standard parabolic subgroup $P_{\Theta \cap w \Omega}$ is equal to $\left(P_\Theta \cap {}^w P_\Omega \right){}^wN_\Omega$.
\item The unipotent radical of $P_{\Theta \cap w \Omega}$ is generated by ${}^wN_\Omega$ and $N_\Theta \cap {}^w N_{\emptyset}$, where $N_{\emptyset}$ is the unipotent radical of the minimal parabolic subgroup corresponding to $\emptyset \subset \Delta_0$. 
\item The standard Levi subgroup of $P_{\Theta \cap w \Omega}$ is $M_{\Theta \cap w \Omega} = M_\Theta \cap w M_\Omega w^{-1}$.
\item The subgroup $P_\Theta \cap {}^wM_\Omega$ is a $w\Omega$-standard parabolic in $M_{w\Omega}={}^wM_\Omega$ with unipotent radical $N_\Theta \cap {}^wM_\Omega$ and standard Levi subgroup $M_{\Theta \cap w \Omega}=M_\Theta \cap {}^w M_\Omega$.
\end{enumerate}
\end{prop}

When applying the Geometric Lemma along two standard parabolic subgroups $P_\Theta$ and $P_\Omega$, associated to $\Theta,\Omega \subset \Delta_0$, we will always use the choice of ``nice" representatives  $[W_\Theta \backslash W_0 / W_\Omega]$ for the double-coset space $P_\Theta \backslash G / P_\Omega$.
\begin{lem}[The Geometric Lemma]\label{geom-lem}
Let $P_\Omega$ and $P_\Theta$ be two $\Delta_0$-standard parabolic subgroups of $G$.
Let $\rho$ be a smooth representation of $M_\Omega$.  
There is a filtration of the space of the representation $(\iota_{P_\Omega}^G \rho)_{N_\Theta}$ such that the associated graded object is isomorphic to the direct sum
\begin{align*}
\bigoplus_{w \in [W_\Theta \backslash W_0 / W_\Omega]} \iota_{M_{\Theta}\cap {}^wP_\Omega}^{M_\Theta} \left( ({}^w\rho)_{N_\Theta \cap{}^w M_\Omega} \right).
\end{align*}
\end{lem}
We write $\mathcal F_{\Theta}^w(\rho) = \mathcal F_{N_\Theta}^w(\rho)$ to denote the representation $\iota_{M_{\Theta}\cap {}^wP_\Omega}^{M_\Theta} \left( ({}^w\rho)_{N_\Theta \cap{}^w M_\Omega} \right)$ of $M_\Theta$. 
\subsection{Distinction of induced representations}
\Cref{lem-hom-injects} is well known and follows from an explicit version of Frobenius Reciprocity due to Bernstein and Zelevinsky \cite[Proposition 2.29]{bernstein--zelevinsky1976}.  Let $Q = LU$ be a $\theta$-stable parabolic subgroup with $\theta$-stable Levi factor $L$ and unipotent radical $U$.  Note that the identity component of $Q^\theta = L^\theta U^\theta$ is a parabolic subgroup of $H^\circ$, with the expected Levi decomposition (\textit{cf.}~\cite{helminck--wang1993}, \cite[Lemma 3.1]{gurevich--offen2015}).  Let $\mu$ be a positive quasi-invariant measure on the (compact) quotient $Q^\theta \backslash H$ \cite[Theorem 1.21]{bernstein--zelevinsky1976}.

\begin{lem}\label{lem-hom-injects}
Let $\rho$ be a smooth representation of $L$ and let $\pi = \iota_Q^G \rho$. 
The map $\lambda \mapsto \lambda^G$ is an injection of $\Hom_{L^\theta}( \delta_Q^{1/2}\rho, \delta_{Q^\theta})$ into $\Hom_H(\pi,1)$, where $\lambda^G$ is given explicitly by
\begin{align*}
\ip{\lambda^G}{\phi} &= \int_{Q^\theta \backslash H} \ip{\lambda}{\phi(h)} \ d\mu(h)
\end{align*}
for any function $\phi$ in the space of $\pi$.  
\end{lem}

\begin{cor}\label{cor-hom-injects}
If $\delta_Q^{1/2}$ restricted to $L^\theta$ is equal to $\delta_{Q^\theta}$, then map $\lambda \mapsto \lambda^G$ is an injection of $\Hom_{L^\theta}(\rho, 1)$ into $\Hom_H(\pi,1)$. In particular, if $\rho$ is $L^\theta$-distinguished, then $\pi$ is $H$-distinguished.
\end{cor}

\begin{proof}
Observe that $\Hom_{L^\theta}(\delta_Q^{1/2}\rho, \delta_{Q^\theta}) = \Hom_{L^\theta}(\rho, \delta_Q^{-1/2}\vert_{L^\theta}\delta_{Q^\theta})$.
\end{proof}

Alternatively, the $H$-invariant linear form on $\pi = \iota_Q^G \rho$ may be understood to arise from the closed orbit in $Q \backslash G / H$ via the Mackey theory.
\subsection{Invariant linear forms on Jacquet modules}\label{sec-r-P-lambda}
Let $(\pi,V)$ be an admissible $H$-distinguished representation of $G$. Let $\lambda$ be a nonzero element of $\Hom_H(\pi,1)$.  Let $P$ be a $\theta$-split parabolic subgroup of $G$ with  unipotent radical $N$ and $\theta$-stable Levi component $M= P \cap \theta(P)$.  
One may associate to $\lambda$ a canonical $M^\theta$-invariant linear form $r_P\lambda$ on the Jacquet module $(\pi_N, V_N)$.
The construction of $r_P\lambda$, via Casselman's Canonical Lifting \cite[Proposition 4.1.4]{Casselman-book}, was discovered independently by Kato--Takano and Lagier.  
We refer the reader to \cite{kato--takano2008, lagier2008} for the details of the construction. We record the following result (\textit{cf.}~\cite[Proposition 5.6]{kato--takano2008}).
\begin{prop}[Kato--Takano, Lagier]\label{rPlambda-prop}
Let $(\pi,V)$ be an admissible $H$-distinguished representation of $G$.
Let $\lambda \in \Hom_H(\pi,1)$ be nonzero and let $P$ be a $\theta$-split parabolic subgroup of $G$ with unipotent radical $N$ and $\theta$-stable Levi component $M= P \cap \theta(P)$.
\begin{enumerate}
\item The linear functional $r_P\lambda: V_N \rightarrow \C$ is $M^\theta$-invariant.
\item The mapping $r_P: \Hom_H(\pi,1) \rightarrow \Hom_{M^\theta}(\pi_N, 1)$, sending $\lambda$ to $r_P\lambda$, is linear.
\end{enumerate}
\end{prop}
Kato and Takano use the invariant forms $r_P\lambda$ to provide the following characterization of relatively supercuspidal representations \cite[Theorem 6.2]{kato--takano2008}.
\begin{thm}[Kato--Takano]\label{kt08-rsc-thm}
Let $(\pi,V)$ be an admissible $H$-distinguished representation of $G$ and let $\lambda$ be a nonzero $H$-invariant linear form on $V$. Then, $(\pi,V)$ is $(H,\lambda)$-relatively supercuspidal if and only if $r_P\lambda = 0$ for every proper $\theta$-split parabolic subgroup $P$ of $G$.
\end{thm}
\subsection{Exponents and the Relative Casselman's Criterion}\label{sec-exp-and-rel-cass}
Let $(\pi, V)$ be a finitely generated admissible representation of $G$.  Recall that $A_G$ denotes the $F$-split component of the centre of $G$.
Let $\chi$ be a quasi-character of $A_G$. For $n\in \N$, $n\geq 1$, define the subspace
\begin{align*}
V_{\chi, n} = \{ v \in V : (\pi(z) - \chi(z))^n v = 0 \ \text{for all} \ z\in A_G \},
\end{align*}
and set 
\begin{align*}
V_{\chi, \infty} = \bigcup_{n=1}^\infty V_{\chi, n}.
\end{align*}
Each $V_{\chi, n}$ is a $G$-stable subspace of $V$ and $V_{\chi, \infty}$ is the generalized eigenspace in $V$ for the $A_G$-action on $V$ by the eigencharacter $\chi$.  
By \cite[Proposition 2.1.9]{Casselman-book}, we have that
\begin{enumerate}
\item $V$ is a direct sum $\displaystyle V = \bigoplus_{\chi} V_{\chi,\infty}$, where $\chi$ ranges over quasi-characters of $A_G$, and
\item since $V$ is finitely generated, there are only finitely many $\chi$ such that $V_{\chi,\infty} \neq 0$. Moreover, there exists $n\in \N$ such that $V_{\chi,\infty} = V_{\chi,n}$, for each $\chi$.
\end{enumerate}
Let $\Exp_{A_G}(\pi)$ be the (finite) set of quasi-characters of $A_G$ such that $V_{\chi,\infty} \neq 0$.  The quasi-characters that appear in $\Exp_{A_G}(\pi)$ are called the exponents of $\pi$.
The second item above implies that $V$ has a finite filtration such that the quotients are $\chi$-representations, for $\chi \in \Exp_{A_G}(\pi)$.
From this last observation, we obtain the following lemma.
\begin{lem}\label{exp-irred-subq}
The characters $\chi$ of $A_G$ that appear in $\Exp_{A_G}(\pi)$ are precisely the central quasi-characters of the irreducible subquotients of $\pi$.
\end{lem}

Note that he same analysis as above can be carried out for any closed subgroup $Z$ of $Z_G$, i.e., we can consider the generalized $Z$-eigenspaces in $V$.  Moreover, we have the following.

\begin{lem}\label{closed-subgp-centre}
Let $Z_1 \supset Z_2$ be two closed subgroups of the centre $Z_G$ of $G$.  The map of exponents $\Exp_{Z_1}(\pi) \rightarrow \Exp_{Z_2}(\pi)$ defined by restriction of quasi-characters is surjective.
\end{lem}

\begin{proof}
Let $\chi \in \Exp_{Z_2}(\pi)$. By assumption, there exists a nonzero vector $v \in V_{\chi, \infty}$. 
In particular, there is an irreducible subquotient of $V_{\chi, \infty}$, hence of $(\pi,V)$, where $Z_2$ acts by the character $\chi$.  On this irreducible subquotient, by Schur's Lemma, the subgroup $Z_1$ must act by some extension $\widehat\chi$ of $\chi$.  
By \Cref{exp-irred-subq}, $\widehat\chi$ must occur in $\Exp_{Z_1}(\pi)$.
\end{proof}

For our purposes, we're interested in the exponents of parabolically induced representations.

\begin{lem}\label{red-to-ind-exp}
Let $P = MN$ be a parabolic subgroup of $G$, let $(\rho, V_\rho)$ be an  finitely generated admissible representation of $M$ and let $\pi = \iota_P^G\rho$.
The quasi-characters $\chi \in \Exp_{A_G}(\pi)$ are the restriction to $A_G$ of characters $\mu$ of $A_M$ appearing in $\Exp_{A_M}(\rho)$.
\end{lem}

\begin{proof}
Without loss of generality, assume that $P=MN$ is a proper parabolic subgroup of $G$.  
Given $a\in A_G$, we have that $\delta_P(a) = 1$, since $a$ is central in $G$.  It follows that for any $f\in V$ we have
\begin{align}\label{action-A_G}
\left(\pi(a)f \right)(g) & = f(ga) 
 = f(ag) 
 = \delta_P^{1/2}(a) \rho(a) f(g) 
 = \rho(a) f(g), 
\end{align}
for all $a\in A_G$ and $g\in G$.
Suppose that $\chi \in  \Exp_{A_G}(\pi)$ and $f\in V_{\chi,\infty}$ is nonzero.  Fix $g_0 \in G$ such that $w_0 = f(g_0)$ is nonzero.  
There exists $n\in \N$, $n\geq 1$ such that $f\in V_{\chi,n}$. 
More precisely, $(\pi(a) - \chi(a))^n f = 0_V$, for all $a\in A_G$, where $0_V: G \rightarrow V_\rho$ is the zero function.  
By induction and using \eqref{action-A_G}, we see that
\begin{align*}
0 & = [(\pi(a) - \chi(a))^n f](g_0) = (\rho(a) - \chi(a))^n (f(g_0)) = (\rho(a) - \chi(a))^n w_0,
\end{align*}
for any $a\in A_G$.  
That is, $w_0 \in (V_\rho)_{\chi,\infty}$ and $(V_\rho)_{\chi,\infty}$ is nonzero; moreover, $\chi \in \Exp_{A_G}(\rho)$.  
By \Cref{closed-subgp-centre}, the map $\Exp_{A_M}(\rho) \rightarrow \Exp_{A_G}(\rho)$ defined by restriction is surjective.
In particular, there exists $\mu \in \Exp_{A_M}(\rho)$ such that $\chi$ is equal to the restriction of $\mu$ to $A_G$.
\end{proof}

Let $(\pi,V)$ be a finitely generated admissible  representation of $G$. Let $P=MN$ be a parabolic subgroup of $G$ with Levi factor $M$ and unipotent radical $N$.  
It is a theorem of Jacquet that $(\pi_N, V_N)$ is also finitely generated and admissible (\textit{cf.}~\cite[Theorem 3.3.1]{Casselman-book}).  Applying \cite[Proposition 2.1.9]{Casselman-book}, we obtain a direct sum decomposition
\begin{align*}
V_N = & \bigoplus_{\chi \in \Exp_{A_M}(\pi_N)} (V_N)_{\chi,\infty}
\end{align*}
where the set $\Exp_{A_M}(\pi_N)$ of quasi-characters of $A_M$, such that $(V_N)_{\chi,\infty} \neq 0$, is finite.
The quasi-characters of $A_M$ appearing in $\Exp_{A_M}(\pi_N)$ are called the exponents of $\pi$ along $P$.

 Suppose, in addition, that $(\pi,V)$ is $H$-distinguished. Fix a nonzero $H$-invariant form $\lambda$ on $V$.
 For any closed subgroup $Z$ of the centre of $G$, Kato and Takano \cite{kato--takano2010} define
 \begin{align}\label{relative-exponent}
 \Exp_Z(\pi,\lambda) & = \{ \chi \in \Exp_Z(\pi) : \lambda \vert_{V_{\chi,\infty}} \neq 0 \},
 \end{align}
 and refer to the set $\Exp_Z(\pi,\lambda)$ as exponents of $\pi$ relative to $\lambda$.
The next result is \cite[Theorem 4.7]{kato--takano2010}, which is a key ingredient used in the proof of our main result.

\begin{thm}[The Relative Casselman's Criterion, Kato--Takano]\label{rel-casselman-crit}
Let $\omega$ be a unitary character of $Z_G$. Let $(\pi,V)$ be a finitely generated admissible $H$-distinguished $\omega$-representation of $G$.
Fix a nonzero $H$-invariant linear form $\lambda$ on $V$. The representation $(\pi,V)$ is $(H,\lambda)$-relatively square integrable  
if and only if the condition 
 \begin{align}\label{rel-casselman} 
|\chi(s)| &< 1 & \text{for all} \ \chi \in \Exp_{S_M}(\pi_N, r_P\lambda) \ \text{and all} \ s\in S_M^- \setminus S_GS_M^1
 \end{align}
  is satisfied for every proper $\theta$-split parabolic subgroup $P=MN$ of $G$.
\end{thm}

\begin{rmk}
Note that the Relative Casselman's Criterion reduces to Casselman's Criterion in the group case: $G = G' \times G'$ and $H = \Delta G \cong G'$ is the diagonal subgroup.
\end{rmk}

\begin{cor}[Kato--Takano]
If $(\pi,V)$ is an $H$-distinguished discrete series representation of $G$, then $\pi$ is $H$-relatively square integrable.
\end{cor}

The next two lemmas let us prove \Cref{reduction-standard-split}, which allows us to reduce to checking the Relative Casselman's Criterion along maximal $\Delta_0$-standard parabolic subgroups (under an additional assumption).

\begin{lem}\label{exponent-bijection}
Let $P = MN$ be a proper ${\theta}$-split parabolic subgroup of $G$.  Assume that $P$ is $H$-conjugate to a $\Delta_0$-standard $\theta$-split parabolic subgroup $P_\Theta$.
If $P = hP_\Theta h^{-1}$, where $h\in H$, then there is a bijection 
\begin{align*}
\Exp_{S_\Theta}(\pi_{N_\Theta}, r_{P_\Theta} \lambda) & \longleftrightarrow \Exp_{S}(\pi_{N}, r_{P} \lambda) \\
\chi' & \mapsto {}^h\chi',
\end{align*}
with inverse given by $\chi \mapsto {}^{h^{-1}}\chi$.
\end{lem}

\begin{proof}
The bijection between $\Exp_{S_\Theta}(\pi_{N_\Theta})$ and $\Exp_S(\pi_N)$ is automatic from the equality $S = h S_\Theta h^{-1}$ (and holds for $h \in (\Hbf \mathbf{M}_0)(F)$).  Using that $h\in H$ and $H$-invariance of $\lambda$, one can show that: If $r_{P_\Theta}\lambda$ is nonzero on $(V_{N_\Theta})_{\chi', \infty}$, then $r_P\lambda$ is nonzero on $(V_N)_{\chi,\infty}$, where $\chi = {}^h\chi'$.
\end{proof}

\begin{lem}\label{lem-red-max-std}
Assume that any ${\theta}$-spilt parabolic subgroup $P$ of $G$ is $H$-conjugate to a $\Delta_0$-standard $\theta$-split parabolic.  If condition \eqref{rel-casselman} holds for all $\Delta_0$-standard ${\theta}$-split parabolic subgroups of $G$, then the condition \eqref{rel-casselman} holds for all ${\theta}$-split parabolic subgroups of $G$.
\end{lem}

\begin{proof}
Let $P = MN$ be a proper ${\theta}$-split parabolic subgroup of $G$.  By assumption, there exists $h\in H$ and a ${\theta}$-split subset $\Theta \subset {\Delta_0}$ such that $P = h P_\Theta h^{-1}$. In particular, the $({\theta},F)$-split component $S$ of $P$ is equal to $hS_\Theta h^{-1}$; moreover, $S^- = h S_\Theta^- h^{-1}$.  Let $\chi \in \Exp_{S}(\pi_{N}, r_{P} \lambda)$, by \Cref{exponent-bijection}, there exists $\chi ' ={}^{h^{-1}}\chi \in \Exp_{S_\Theta}(\pi_{N_\Theta}, r_{P_\Theta} \lambda)$.
Let $s \in S^- \setminus S^1 S_{\Delta_0}$, then $s' = h^{-1} s h \in S_\Theta^- \setminus S_\Theta^1 S_{\Delta_0}$.  It follows that
\begin{align*}
|\chi(s)| & = |\chi(h s' h^{-1})| = |\chi' (s')| <1,
\end{align*}
where the final inequality holds by the assumption that \eqref{rel-casselman} holds for $P_\Theta$.  
\end{proof}

\begin{prop}\label{reduction-standard-split}
Let $\pi$ be an $H$-distinguished representation of $G$ and let $\lambda$ be a nonzero $H$-invariant linear form on the space of $\pi$.
Assume that any ${\theta}$-spilt parabolic subgroup $P$ of $G$ is $H$-conjugate to a $\Delta_0$-standard $\theta$-split parabolic.  
Then $\pi$ is $(H,\lambda)$-relatively square integrable if and only if the condition \eqref{rel-casselman} holds for all $\Delta_0$-standard maximal ${\theta}$-split parabolic subgroups of $G$.
\end{prop}

\begin{proof}
Apply \Cref{lem-red-max-std}, \cite[Lemma 4.6]{kato--takano2010} and \Cref{rel-casselman-crit}.
\end{proof}

The next result is key in our application of \Cref{rel-casselman-crit}. It allows us to ignore ``bad" exponents relative to $\lambda$, as long as the appropriate subquotients are not distinguished.

\begin{prop}\label{non-dist-gen-eig-sp}
Let $(\pi, V)$ be a finitely generated admissible representation of $G$.  Let $\chi \in \Exp_{Z_G}(\pi)$ and assume that none of the irreducible subquotients of $(\pi,V)$ with central character $\chi$ are $H$-distinguished.  Then for any $\lambda \in \Hom_H(\pi,1)$, the restriction of $\lambda$ to $V_{\chi,\infty}$ is equal to zero, i.e., $\lambda \vert_{V_{\chi,\infty}} \equiv 0$.
\end{prop}

\begin{proof}
Suppose, by way of contradiction, that $\lambda \vert_{V_{\chi,\infty}}\neq 0$.  Then $(\pi\vert_{V_{\chi,\infty}}, V_{\chi,\infty})$ is an admissible finitely generated $H$-distinguished representation of $G$. By \Cref{sub-quotient2}, some irreducible subquotient $(\rho, {V_\rho})$ of $V_{\chi,\infty}$ must be $H$-distinguished.  However, the representation $(\rho, {V_\rho})$ is also an irreducible subquotient of $(\pi,V)$ and has central character $\chi$.  By assumption, no such $(\rho, {V_\rho})$ can be $H$-distinguished; therefore, we must have that $\lambda \vert_{V_{\chi,\infty}}$ is identically zero.
\end{proof}

\section{Tori and parabolic subgroups: The linear and Galois cases}\label{sec-structure-lin-Gal}
\begin{rmk}
For the remainder of the paper we work in the linear and Galois cases.
Refer to Sections \ref{sec-notation-lin} and \ref{sec-notation-Gal} for notation.
\end{rmk}

\subsection{Tori and root systems relative to $\theta$}\label{sec-tori-roots}
In the linear case, let $A_0$ be the diagonal maximal $F$-split torus of $G$. 
Note that $A_0$ is $\theta$-stable. 
Let $S_0$ be the $(\theta,F)$-split component of $A_0$.
It is straightforward to check that 
\begin{align*}
S_0 &= \{\diag(a_1,\ldots, a_{\frac{n}{2}}, a_{\frac{n}{2}}^{-1}, \ldots, a_1^{-1}) : a_i \in F^\times, 1\leq i \leq \textstyle\frac{n}{2}\}.
\end{align*}
Moreover, $S_0$ is a maximal $(\theta,F)$-split torus of $G$.
Indeed, it is readily verified that the upper-triangular Borel subgroup of $G$ is a minimal $\theta$-split parabolic subgroup with Levi subgroup $A_0$. It follows from \cite[Proposition 4.7(\rm{iv})]{helminck--wang1993} that $S_0$ is a maximal $(\theta,F)$-split torus of $G$ contained in $A_0$.
In the Galois case, the torus $\mathbf T$ obtained as the restriction of scalars of the diagonal torus of $\GL_n$ is a maximal non-split $F$-torus of $\G$. 
We identify $T = \mathbf{T}(F)$ with the diagonal matrices in $\GL_n(E)$. 
Define $T_0 = {}^\gamma T$, where $\gamma$ is described in $\S$\ref{sec-choice-elts}. Then $A_0={}^\gamma A_T$ is the $F$-split component  of $T_0$.
The tori $T$, $A_T$, $T_0$ and $A_0$ are all ${\theta}$-stable.
As above, and using \eqref{Galois-involution-relation}, it is readily verified that
\begin{align*}
S_0 =    \{ {}^\gamma \diag(a_1,\ldots, a_{\floor{\frac{n}{2}}}, \widehat 1, a_{\floor{\frac{n}{2}}}^{-1}, \ldots, a_1^{-1} ) : a_i \in F^\times, 1\leq i \leq \floor{\textstyle\frac{n}{2}}\},
\end{align*} 
is a maximal $(\theta,F)$-split torus of $G$ contained in $A_0$.

In both cases, let $\Phi_0 = \Phi(G,A_0)$ be the set of roots of $G$ relative to $A_0$.  
Explicitly, in the linear case, we have
\begin{align*}
\Phi_0 = \{ \epsilon_i - \epsilon_j : 1\leq i\neq j \leq n\},
\end{align*}
where $\epsilon_i \in X^*(A_0)$ is the $i$\textsuperscript{th} coordinate ($F$-rational) character of $A_0$.
Let 
\begin{align*}
\Delta_0 = \{\epsilon_i - \epsilon_{i+1} : 1 \leq i \leq n-1\}
\end{align*}
 be the standard base of $\Phi_0$.
 The set $\Phi_0^+$ of positive roots (determined by $\Delta_0$) is 
\begin{align*}
\Phi_0^+ =  \{ \epsilon_i - \epsilon_j : 1\leq i < j \leq n\}.
\end{align*}
 In the Galois case, we relate $\Phi_0$ to another collection of roots, those relative to $A_T$.
Let $\Phi = \Phi(G, A_T)$ be the root system of $G$ with respect to $A_T$ with standard base $\Delta$.
We observe that $\Phi_0 = {}^\gamma \Phi$, where given a root $\beta \in \Phi$ we have
\begin{align*}
({}^\gamma\beta)(a) = \beta ({}^{\gamma^{-1}}a) = \beta (\gamma^{-1} a \gamma),
\end{align*}
for $a\in A_0$. Moreover, $\Delta_0 = {}^\gamma \Delta$ is a base for $\Phi_0$ and it is clear that 
\begin{align*}
{\Phi}_0^+ = \{ {}^\gamma (\epsilon_i - \epsilon_j) : 1\leq i < j \leq n\},
\end{align*}
where, as above, $\epsilon_i$ is the $i$\textsuperscript{th}-coordinate ($F$-rational) character of the diagonal $F$-split torus $A_T$.
It is elementary to verify the following.
\begin{lem}\label{lem-Delta-0-theta-base}
The  set of simple roots $\Delta_0$ of $\Phi_0$ is a ${\theta}$-base for $\Phi_0$. In addition, the subset of ${\theta}$-fixed roots in $\Phi_0$ is empty.
\end{lem}

\begin{cor}\label{cor-min-std-split-pblc}
The Borel subgroup $P_\emptyset = P_0 = M_0 N_0$  corresponding to $\emptyset \subset \Delta_0$ is a minimal $\theta$-split parabolic subgroup of $G$. 
\end{cor}

\begin{proof}
The subset $\Delta_0^\theta$ is a minimal $\theta$-split subset of $\overline{\Delta}_0$; therefore, the parabolic $P_{\Delta_0^\theta}$ is a minimal standard $\theta$-split parabolic subgroup \cite{kato--takano2010}.  
Since  $\Delta_0^\theta=\emptyset$, we have  $P_{\Delta_0^\theta} = P_\emptyset = P_0$. In the linear case, $M_0 = A_0$ and in the Galois case $M_0 = C_G(A_0) = T_0$.
\end{proof}

Following $\S$\ref{sec-tori-involution}, since $\Delta_0^\theta = \emptyset$, the restricted root system is just the image of $\Phi_0$ under the restriction map 
$p: X^*(A_0) \rightarrow X^*(S_0)$. That is, we have $\overline \Phi_0 = p(\Phi_0)$ and $\overline \Delta_0 = p(\Delta_0)$. 
Explicitly, in the linear case,
\begin{align*}
\overline \Delta_0 = \left\{\bar\epsilon_i - \bar\epsilon_{i+1} : 1\leq i \leq \textstyle\frac{n}{2}-1\right\} \cup \left\{2\bar\epsilon_{\frac{n}{2}}\right\},
\end{align*}
where $\bar\epsilon_i \in X^*(S_0)$ is the $i$\textsuperscript{th} coordinate character of $S_0$ given by
\begin{align*}
\bar\epsilon_i (\diag(a_1,\ldots,a_{\frac{n}{2}},a_{\frac{n}{2}}^{-1},\ldots, a_1^{-1}))= a_i.
\end{align*}
Similarly in the Galois case, we have that
\begin{align*}
\overline \Delta_0= \left\{{}^\gamma\bar\epsilon_i - {}^\gamma\bar\epsilon_{i+1} : 1\leq i \leq \floor{\textstyle\frac{n}{2}}-1\right\} \cup \{\bar\alpha\},
\end{align*}
where $\bar \alpha = {}^\gamma\bar\epsilon_{\floor{\frac{n}{2}}}$ when $n$ is odd, and $\bar \alpha = 2{}^\gamma\bar\epsilon_{\frac{n}{2}}$ when $n$ is even.
The following result is now an immediate consequence of \Cref{lem-Delta-0-theta-base}.

\begin{lem}\label{lem-max-split-subsets}
For $1\leq k \leq \floor{\frac{n}{2}}$, let $\Theta _k$ denote the $\floor{\frac{n}{2}}$ maximal $\theta$-split subsets of $\Delta_0$.
In the linear case, for $1\leq k \leq n/2-1$
\begin{align*}
\Theta _k & = [\overline \Delta_0 \setminus \{\bar\epsilon_k - \bar\epsilon_{k+1}\}] 
 =  \Delta_0 \setminus \{ \epsilon_k - \epsilon_{k+1}, \epsilon_{n-k} - \epsilon_{n-k+1}\}
\end{align*}
and 
\begin{align*}
\Theta _{\frac{n}{2}} & = [\overline \Delta_0 \setminus \left\{2\bar\epsilon_{\frac{n}{2}} \right\}] 
 =  \Delta_0 \setminus \left\{\epsilon_{\frac{n}{2}} - \epsilon_{\frac{n}{2}+1}\right\}.
\end{align*}
Respectively, in the Galois case, for $1\leq k \leq \floor{\frac{n}{2}}-1$
\begin{align*}
\Theta _k & = [\overline \Delta_0\setminus \{{}^\gamma\bar\epsilon_k - {}^\gamma\bar\epsilon_{k+1}\}] 
 =  {\Delta_0} \setminus \{ {}^\gamma(\epsilon_k - \epsilon_{k+1}), {}^\gamma(\epsilon_{n-k} - \epsilon_{n-k+1})\}
\end{align*}
and 
\begin{align*}
\Theta _{\floor{\frac{n}{2}}} & = [\overline \Delta_0\setminus \{\bar\alpha \}] 
 = {\Delta_0} \setminus p^{-1}\{\bar\alpha\},
\end{align*}
where $\bar \alpha = {}^\gamma\bar\epsilon_{\floor{\frac{n}{2}}}$ when $n$ is odd, and $\bar \alpha = 2{}^\gamma\bar\epsilon_{\frac{n}{2}}$ when $n$ is even.
\end{lem}

\begin{note}
When $n$ is odd 
\begin{align*}
\Theta _{\floor{\frac{n}{2}}} &= {\Delta_0} \setminus \left\{{}^\gamma\left(\epsilon_{\floor{\frac{n}{2}}} - \epsilon_{\floor{\frac{n}{2}}+1}\right), {}^\gamma\left(\epsilon_{\floor{\frac{n}{2}}+1} - \epsilon_{\floor{\frac{n}{2}}+2}\right)\right\},
\end{align*}
 and when $n$ is even 
 \begin{align*}
 \Theta _{\frac{n}{2}} = {\Delta_0} \setminus \left\{ {}^\gamma\left(\epsilon_{\frac{n}{2}} - \epsilon_{\frac{n}{2}+1}\right)\right\}.
 \end{align*}
\end{note}

The next proposition follows immediately from \Cref{lem-max-split-subsets}.

\begin{prop}\label{max-theta-split}
The $\Delta_0$-standard maximal ${\theta}$-split parabolic subgroups of $G$ are:
\begin{align*}
P_k& := P_{\Theta _k} =  \left \{ \begin{array}{lll}  P_{(k,n-2k,k)}, & 1 \leq k \leq \textstyle\frac{n}{2}-1, & \text{in the linear case} \\
						{}^\gamma P_{(k,n-2k,k)},  & 1 \leq k \leq \floor{\frac{n}{2}}-1, & \text{in the Galois case}
	\end{array} \right.
\end{align*}
and
\begin{align*}
P_{\floor{\frac{n}{2}}} & :=  P_{\Theta _{\floor{\frac{n}{2}}}} =  \left \{ \begin{array}{lll}  P_{\left(\frac{n}{2},\frac{n}{2}\right)}, & \text{in the linear case, since $n$ is even} \\
						{}^\gamma P_{\left(\frac{n}{2},\frac{n}{2}\right)},  & \text{in the Galois case when $n$ is even} \\
						{}^\gamma P_{\left(\floor{\frac{n}{2}},1,\floor{\frac{n}{2}}\right)},  & \text{in the Galois case when $n$ is odd}.
	\end{array} \right.
\end{align*}  
\end{prop}

\begin{rmk}\label{rmk-theta-split-parabolic-notation}
In both cases, $P_k = M_k N_k$, where $M_k=M_{\Theta_k}$ is the standard Levi factor and $N_k$ is the unipotent radical of $P_k$. We write $A_k$ for the $F$-split component and $S_k$ for the $({\theta},F)$-split component of $M_k$. 
\end{rmk}

In preparation for our proof of \Cref{no-dist-unit-exp}, here we determine the $\theta$-fixed points of the Levi subgroups $M_k$, $1\leq k \leq \floor{\frac{n}{2}}$.
Recall from \eqref{eq-vartheta-defn} that, in the Galois case, $\vartheta$ is the involution $\gamma\cdot {\theta} = \Int w_\ell \circ {\theta} = {\theta} \circ \Int w_\ell$. 
The following is a special case of \Cref{orbit-dist} and also holds for $\theta_r$ and $\vartheta_r$ (\textit{cf.}~\eqref{vartheta-r-defn}).  
\begin{lem}\label{theta--theta--fixed}
Assume that we are in the Galois case.
An element of $G$ of the form ${}^\gamma x$ is ${\theta}$-fixed (respectively ${\theta}$-split) if and only if $x$ is $\vartheta$-fixed (respectively $\vartheta$-split).  
\end{lem}

\begin{prop}\label{prop-M-Theta-fixed-pts} 
Let $1\leq k \leq \floor{\frac{n}{2}}$.
In the linear case, the group $M_k^\theta$ of $\theta$-fixed points in $M_k = M_{(k,n-2k,k)}$ is equal to
\begin{align}\label{H-bullet}
H_{(k,n-2k,k)} =  \left \{ \begin{array}{ll} \left\{ \diag\left( A, B, \theta_k(A)\right) : A \in G_k,\ B \in H_{n-2k} \right \}, &\text{if} \ k \neq n/2 \\
			 \left\{ \diag\left( A, \theta_k(A)\right) : A \in G_k \right \}, &\text{if} \ k=n/2
			 \end{array} \right.. 
\end{align}
In the Galois case, $M_k = {}^\gamma M_{(k,n-2k,k)}$ and $M_k^\theta$ is $G$-conjugate to $H_{(k,n-2k,k)}$. Explicitly, we have that
\begin{align*}
H_{(k,n-2k,k)} & = \gamma_{(k,n-2k,k)} M_{(k,n-2k,k)}^{\vartheta}\gamma_{(k,n-2k,k)}^{-1}
\end{align*}
and
\begin{align*}
M_k^{\theta} & = \gamma M_{(k,n-2k,k)}^{\vartheta}\gamma^{-1} = \gamma\gamma_{(k,n-2k,k)}^{-1} H_{(k,n-2k,k)}   \gamma_{(k,n-2k,k)}\gamma^{-1},
\end{align*}
where $\gamma_{(k,n-2k,k)}=\diag(\gamma_k, \gamma_{n-2k}, \gamma_k) \in M_{(k,n-2k,k)}$.
\end{prop}

\begin{proof}
Let $1\leq k \leq \floor{\frac{n}{2}}$ and let $M_\bullet = M_{(k,n-2k,k)}$, respectively $M_{(n/2,n/2)}$ when $n$ is even and $k=n/2$.  In the linear case, $M_k = M_\bullet$ while, in the Galois case, $M_k = {}^\gamma M_\bullet$.
The statement in the linear case follows from the proof in the Galois case (note the relationship between $\theta$ and $\vartheta$, \textit{cf.}~\eqref{Galois-involution-relation}).
Without loss of generality, we work in the Galois case and assume that $k < n/2$.
By \Cref{theta--theta--fixed}, we have $M_k^\theta = {}^\gamma \left(M_\bullet^{\vartheta}\right)$.
Let ${}^\gamma m \in M_k$ where $m \in  M_{\bullet}$.  Explicitly, we have $m = \diag(A,B,C)$, where $A,C \in G_k$ and $B \in G_{n-2k}$. 
One may verify that 
\begin{align*}
\vartheta(m)&= \Int w_\ell \circ {\theta} (m) =  \diag(\vartheta_k (C), \vartheta_{n-2k}(B), \vartheta_k(A)).
\end{align*} 
It follows that
\begin{align*}
M_\bullet^{\vartheta}= \left\{ \diag\left( A, B, \vartheta_k(A)\right) : A \in G_k,\ B \in G_{n-2k},\ B = \vartheta_{n-2k}(B) \right \}.
\end{align*}
Conjugating $M_\bullet^{\vartheta}$ by the element $\gamma_\bullet=\diag(\gamma_k, \gamma_{n-2k}, \gamma_k)$ in $M_\bullet$ and applying \Cref{theta--theta--fixed}, we obtain that $M_\bullet^{\vartheta}$ is $M_\bullet$-conjugate (and $F$-isomorphic to) the subgroup $H_\bullet$.
\end{proof}

The next result will allow us to apply \Cref{reduction-standard-split}.

\begin{lem}\label{split-parabolic-conjugate}
Any $\theta$-split parabolic subgroup $P$ of $G$ is $H$-conjugate to a $\Delta_0$-standard $\theta$-split parabolic subgroup $P_\Theta$, for some $\Theta \subset \Delta_0$. 
\end{lem}

\begin{proof}
One can check that the degree-one Galois cohomology of $\mathbf A_0 \cap \Hbf$ (respectively, $\mathbf T_0 \cap \Hbf$) over $F$ is trivial.
By a standard argument, we have that $( \Hbf \mathbf A_0)(F) = HA_0$ (respectively, $(\Hbf\mathbf T_0)(F) = HT_0$). 
The proposition follows from \Cref{cor-min-std-split-pblc} and \cite[Lemma 2.5(2)]{kato--takano2008}.
\end{proof}

\subsection{A class of $\theta$-elliptic Levi subgroups and $\theta$-stable parabolic subgroups}\label{sec-theta-ell-levi}
The next two lemmas may be readily verified by hand.
\begin{lem}\label{min-rel-ell-Levi}
The Levi subgroup $L_0=C_G\left((A_0^\theta)^\circ\right)$ of $G$ is $\theta$-elliptic and $A_0^\theta = (A_0^\theta)^\circ = A_{L_0}$. 
Moreover, $L_0$ is minimal among $\theta$-elliptic Levi subgroup of $G$ that contain $A_0$.
\end{lem}

\begin{proof}
First, we observe that since $(A_0^\theta)^\circ$ is $\theta$-stable, the Levi subgroup $L_0$ is ${\theta}$-stable.
It is immediate that the maximal $F$-split torus $A_0$ is contained in $L_0$ (since $A_0$ is abelian).    

Now, we show that $L_0$ is $\theta$-elliptic.  
First, note that the $(\theta,F)$-split component $S_G$ of $G$ is the trivial group.
Indeed, in the linear case, $\theta$ is inner and we have that $A_G \cong F^\times$ is pointwise $\theta$-fixed. 
It follows from \eqref{eq-S-M-defn} that $S_G = \left( \{\pm e \} \right)^\circ = \{e\}$.
Again, in the Galois case, ${\theta}$ acts trivially on the $F$-split component of the centre $A_G$ of $G$ and $S_G = \{e\}$.
In both the linear and Galois cases, it is readily verified that the $F$-split component of the centre of $L_0$ is equal to $(A_0^\theta)^\circ$, that is, $A_{L_0}=(A_0^\theta)^\circ$. Moreover, in both cases, we  have that $(A_0^\theta)^\circ = A_0^\theta$.
In particular, $A_{L_0}$ is contained in $H$ and it follows that $S_{L_0} = \{e\}$ (\textit{cf.}~\cite[$\S1.3$]{helminck1997}).
By Lemma \ref{theta-elliptic-centre}, $L_0$ is $\theta$-elliptic.

Finally, we prove that $L_0$ is minimal among $\theta$-elliptic Levi subgroups containing $A_0$.  Suppose that $L \subset L_0$ is a proper Levi subgroup of $L_0$ that contains $A_0$.   We argue that $L$ cannot be $\theta$-elliptic.  Since $L$ is proper in $L_0$, we have that $A_{L_0} = (A_0^\theta)^\circ$ is a proper sub-torus of $A_L$.  Following \cite[$\S1.3$]{helminck1997}, we have an almost direct product $A_L = (A_L^\theta)^\circ S_L$.  Observe that, since $A_L \subset A_0$, we have
\begin{align*}
(A_L^\theta)^\circ = (A_L \cap A_0^\theta)^\circ = (A_L \cap A_{L_0})^\circ = A_{L_0}.
\end{align*}
Since $S_G =S_{L_0} = \{e\} \subset A_{L_0} = (A_L^\theta)^\circ$ and $A_L = (A_L^\theta)^\circ S_L$ properly contains $A_{L_0}$, it must be the case that $S_G$ is a proper subtorus of $S_L$; in particular, $S_L$ is non-trivial.  It follows from Lemma \ref{theta-elliptic-centre} that $L$ is not $\theta$-elliptic and this completes the proof.
\end{proof}

\begin{lem}\label{Weyl-gp-conjugate}
In the Galois case, conjugation by $\gamma$ maps $N_{G}(A_T)$ to $N_{G}(A_0)$ and induces an explicit isomorphism of the Weyl group $W_T = W(G,A_T)$, with respect to $A_T$, with the Weyl group $W_0 = W(G,A_0)$, with respect to $A_0$. Moreover, we identify $W_T$ with the group of permutation matrices in $G$ (isomorphic to the symmetric group $\style S_n$) and $W_0$ with the $\gamma$-conjugates of the permutation matrices.
\end{lem}

In both the linear and Galois cases, define 
\begin{align*}
{\Delta^{ell}} = w_0 \Delta_0,
\end{align*}
where $w_0$ is defined in \eqref{w-zero}.
Since $\Delta^{ell}$ is a Weyl group translate of $\Delta_0$, we have that $\Delta^{ell}$ is a base of  $\Phi_0$.
In both cases, set
\begin{align*}
\Delta_{odd} = \{\epsilon_{i}- \epsilon_{i+1} : i\ \text{is odd}\},
\end{align*} 
and in the Galois case further denote 
\begin{align*}
\Delta_{0,odd} = {}^\gamma \Delta_{odd} \subset \Delta_0 = {}^\gamma \Delta.
\end{align*}  
 In the linear case, the define the subset $\Delta^{ell}_{\min}$ of ${\Delta^{ell}}$ by 
\begin{align*}
{\Delta^{ell}_{\min}} = w_0 \Delta_{odd}
\end{align*}
and in the Galois case, define 
\begin{align*}
{\Delta^{ell}_{\min}} = w_0 \Delta_{0,odd}.
\end{align*} 
In both cases, the subset $\Delta^{ell}_{\min}$ is exactly the subset of $\Delta^{ell}$ that cuts out the torus $A_0^{\theta}$ from $A_0$.  
In particular,
 \begin{align}\label{A0-ell-roots}
 A_0^{\theta} = A_{{\Delta^{ell}_{\min}}} = \left( \bigcap_{\beta \in {\Delta^{ell}_{\min}}} \ker (\beta: A_0 \rightarrow F^\times) \right)^\circ,
 \end{align}
 and 
 \begin{align*}
 L_0 = L_{{\Delta^{ell}_{\min}}} =C_{G}\left(A_{{\Delta^{ell}_{\min}}}\right).
 \end{align*}
Write $\Delta_{even} = \{ \epsilon_i - \epsilon_{i+1} : 2\leq i\leq n-1,\ i \ \text{is even} \}$.

\begin{defn}\label{defn-std-theta-ell-levi}
A Levi subgroup $L$ of $G$ is a standard-$\theta$-elliptic Levi subgroup if and only if $L$ is $\Delta^{ell}$-standard and contains $L_0$.  
\end{defn}

The next proposition characterizes the inducing subgroups in \Cref{main-RDS-thm}.

\begin{prop}\label{ind-subgp-prop}
Let $\Omega^{ell} \subset {\Delta^{ell}}$ such that $\Omega^{ell}$ contains $\Delta^{ell}_{\min}$. 
\begin{enumerate}
\item The ${\Delta^{ell}}$-standard parabolic subgroup $Q = Q_{\Omega^{ell}}$, associated to $\Omega^{ell}$, is ${\theta}$-stable.
\item In particular, the unipotent radical $U = U_{\Omega^{ell}}$ is $\theta$-stable.  
\item The Levi subgroup $L=L_{\Omega^{ell}} = C_G(A_{\Omega^{ell}})$ is a standard-$\theta$-elliptic Levi of $G$. 
\item The modular function $\delta_Q$ of $Q$ satisfies
$\left. \delta_Q^{1/2} \right\vert_{L^{\theta}} = \delta_{Q^{\theta}}.$ \label{ind-subgp-prop-modular}
\item We have that 
\begin{align}\label{L-isom-class}
L &\cong \prod_{i=1}^k G_{m_i} &\text{and}& &L^{\theta} &\cong \prod_{i=1}^k H_{m_i}
\end{align}
where $\sum_{i=1}^k m_i = n$, such that when $n$ is odd exactly one $m_i$ is odd, and when $n$ is even all of the $m_i$ are even.
\end{enumerate}
\end{prop}

\begin{proof}
Since $\Delta_{\min}^{ell} \subset \Omega^{ell}$, it follows from \eqref{A0-ell-roots} that $A_L = A_{\Omega^{ell}}$ is contained in $A_{L_0} = A_0^\theta$.  
It follows that $L_0$ is contained in $L$ and $L$ is a standard-$\theta$-elliptic Levi subgroup by \Cref{contain-theta-elliptic} and \Cref{min-rel-ell-Levi}. 
By \Cref{theta-elliptic-theta-stable}, $Q$ is a $\theta$-stable parabolic subgroup with $\theta$-stable unipotent radical $U = U_{\Omega^{ell}}$.

In the Galois case, the statement about modular functions is \cite[Lemma 5.5]{gurevich--offen2015}, a proof is given in \cite[Lemma 2.5.1]{lapid--rogawski2003}.  In the linear case, one may compute the modular functions by hand to verify the desired equality. We omit the straightforward computation.

Finally, we explicitly describe both $L$ and $L^\theta$. Note that, in the Galois case, $\gamma$ centralizes $A_T^\vartheta$ and by \Cref{theta--theta--fixed} we see that $A_0^\theta = \gamma A_T^\vartheta \gamma^{-1} = A_T^\vartheta$. 
In both cases, it follows that $A_0^\theta$ is equal to the $w_+$-conjugate of the $F$-split torus
\begin{align*}
A_{(2,\ldots,2,\widehat1)} = \{ \diag(a_1,a_1,a_2,a_2, \ldots, a_{\floor{\frac{n}{2}}},a_{\floor{\frac{n}{2}}}, \widehat{a}) : a, a_i \in F^\times \},
\end{align*}
corresponding to the partition $(2,\ldots,2,\widehat1)$ of $n$.
Precisely, $A_0^\theta = w_+ A_{(2,\ldots,2,\widehat1)} w_+^{-1}$; moreover, it follows that $L_0 = w_+ M_{(2,\ldots,2,\widehat1)} w_+^{-1}$.
Furthermore, we can realize $\Omega^{ell} = w_0 \Omega$, where $\Omega \subset \Delta_0$ contains $\Delta_{odd}$, respectively $\Delta_{0,odd}$.
In the linear case, $\Omega^{ell} = w_+ \Omega$ since $w_0 = w_+$, while in the Galois case, $\Omega^{ell} = w_0 \Omega = w_+ \gamma \underline \Omega$, where $\underline \Omega = {}^{\gamma^{-1}} \Omega \subset \Delta$ contains $\Delta_{odd}$. 
It follows that $L = L_{\Omega^{ell}}$ is the $w_+$-conjugate of a block diagonal Levi subgroup $M_{(m_1,\ldots, m_k)}$ that contains $M_{(2,\ldots,2,\widehat1)}$.
We have now established the first isomorphism in \eqref{L-isom-class}:
\begin{align*}
L & = w_+ M_{(m_1,\ldots, m_k)} w_+^{-1} \cong \prod_{i=1}^k G_{m_i},
\end{align*}
where the partition $(m_1,\ldots, m_k)$ of $n$ is refined by $(2,\ldots,2,\widehat1)$.  
In particular, when $n$ is even, each $m_i$ is even and when $n$ is odd, exactly one $m_j$ is odd.

Let $l = w_+ m w_+^{-1} \in L$, where $m \in M=M_{(m_1,\ldots, m_k)}$.  To determine $L^\theta$, we treat the linear and Galois cases separately.  
Starting in the linear case, we see that $l$ is $\theta$-fixed if and only if $m$ is fixed by the involution $\theta_+ = w_+ \cdot \theta = \Int({w_+^{-1}w_\ell w_+})$.  
The element $w_+^{-1}w_\ell w_+$ is the permutation matrix
\begin{align*}
w_+^{-1}w_\ell w_+ = \left( \begin{matrix} \begin{matrix}0 & 1 \\ 1 &0 \end{matrix} & & \\
								& \ddots& \\
								& & \begin{matrix} 0& 1 \\ 1 & 0\end{matrix} 	\end{matrix} \right),
\end{align*}
which lies in $M_{(2,\ldots,2)} \subset M$.  
We observe that, since each $m_i$ is even, $\theta_+$ acts on the $i$\textsuperscript{th}-block $G_{m_i} = \GL_{m_i}(F)$ of $M$ as conjugation by 
\begin{align*}
w_{m_i} = \left( \begin{matrix} \begin{matrix}0 & 1 \\ 1 &0 \end{matrix} & & \\
								& \ddots& \\
								& & \begin{matrix} 0& 1 \\ 1 & 0\end{matrix} 	\end{matrix} \right) \in G_{m_i}.
\end{align*}
Moreover, $w_{m_i}$ is $G_{m_i}$-conjugate to $J_{m_i}$ (\textit{cf.}~$\S$\ref{sec-choice-elts}).
It follows that $\theta_+$ acting on $M$ is $M$-equivalent to the product involution $\theta_{m_1} \times \ldots \times \theta_{m_k}$; 
therefore, by \Cref{orbit-dist} we have
\begin{align*}
L^{\theta} & = w_+ M^{\theta_+} w_+^{-1} \cong M^{\theta_+} \cong \prod_{i=1}^k (G_{m_i})^{\theta_{m_i}} = \prod_{i=1}^k H_{m_i},
\end{align*}
where the second isomorphism is given by conjugation by an element of $M$.
In the Galois case, note that $w_+$ is ${\theta}$-fixed and $M$ is ${\theta}$-stable.  
Then $l = w_+ m w_+^{-1}$ is ${\theta}$-fixed if and only if $m$ is ${\theta}$-fixed.
It follows that $L^{\theta} = w_+ M^{\theta} w_+^{-1}$ and we have that
\begin{align*}
L^{\theta} = w_+ M^{\theta} w_+^{-1} \cong M^{\theta} = \prod_{i=1}^k (G_{m_i})^{{\theta}_{m_i}}= \prod_{i=1}^k H_{m_i},
\end{align*}
as claimed. 
\end{proof}

\section{Construction of relative discrete series: The main theorem}\label{sec-main-theorem}
\begin{defn}\label{defn-regular}
A smooth representation $\tau$ of a Levi subgroup $L$ of $G$ is regular if for every non-trivial element $w \in N_G(L)/L$ we have that ${}^w \tau \ncong \tau$, where ${}^w \tau = \tau \circ \Int w^{-1}$.
\end{defn}

For general linear groups, we can immediately translate \Cref{defn-regular} into the following result.

\begin{lem}\label{pairwise-inequiv}
Let $(m_1,\ldots, m_k)$ be a partition of $n$.
Let $\tau_i$ be an irreducible admissible representation of $G_{m_i}$, for $1 \leq i \leq k$.
The representation $\tau_1 \otimes \ldots \otimes \tau_k$ of $M_{(m_1,\ldots, m_k)}$ is regular if and only if $\tau_i \ncong \tau_j$, for all $1\leq i \neq j \leq k$.
\end{lem}

Now we come to the main result of the paper.

\begin{thm}\label{main-RDS-thm}
Let $Q=LU$ be a proper $\Delta^{ell}$-standard $\theta$-stable parabolic subgroup of $G$ with standard-$\theta$-elliptic Levi factor $L$ and unipotent radical $U$.
Let $\tau$ be a regular $L^\theta$-distinguished discrete series representation of $L$.  
The parabolically induced representation $\pi = \iota_Q^G \tau$ is irreducible and $H$-relatively square integrable.
\end{thm}

\begin{proof}
By assumption, $\tau$ is unitary and regular; therefore, $\pi$ is irreducible by a result of Bruhat \cite{bruhat1961} (\textit{cf.}~\cite[Theorem 6.6.1]{Casselman-book}).
Since $\tau$ is $L^\theta$-distinguished, $\pi$ is $H$-distinguished by \Cref {ind-subgp-prop}\eqref{ind-subgp-prop-modular} and \Cref{cor-hom-injects}.
Let $\lambda$ denote a fixed nonzero $H$-invariant linear form on $\pi$.
By \Cref{prop-mult-one}, $\lambda$ is unique up to scalar multiples.  
To complete the proof, it remains to show that $\pi$ satisfies the Relative Casselman's Criterion.

By \Cref{split-parabolic-conjugate} and \Cref{reduction-standard-split}, it is sufficient to verify that the condition \eqref{rel-casselman} is satisfied for every $\Delta_0$-standard maximal ${\theta}$-split parabolic subgroup.
By assumption, $Q = Q_{\Omega^{ell}}$ for some proper subset $\Omega^{ell} = w_0 \Omega$ of $\Delta^{ell}$ containing $\Delta^{ell}_{\min}$, where $\Omega \subset \Delta_0$.  Let $P_\Theta$ be a maximal $\Delta_0$-standard $\theta$-split parabolic subgroup. It follows from \Cref{lem-nice-reps} that the set $[W_\Theta \backslash W_0 / W_\Omega]\cdot w_0^{-1}$ provides a ``nice" choice of representatives for the double-coset space $P_\Theta \backslash G / Q$.

 By the Geometric Lemma (\textit{cf.}~\Cref{geom-lem}) and \Cref{exp-irred-subq}, the exponents of $\pi$ along $P_\Theta$ are given by the union
\begin{align*}
\Exp_{A_\Theta}(\pi_{N_\Theta}) = \bigcup_{y\in [W_\Theta \backslash W_0 / W_\Omega]\cdot w_0^{-1}} \Exp_{A_\Theta} (\mathcal F_\Theta^y(\tau)),
\end{align*} 
where the exponents on the right-hand side are the central characters of the irreducible subquotients of $\mathcal F_\Theta^y(\tau)$.
By \Cref{closed-subgp-centre}, the map from $\Exp_{A_\Theta} (\mathcal F_\Theta^y(\tau))$ to $\Exp_{S_\Theta}(\pi_{N_\Theta}, r_{P_\Theta} \lambda)$ defined by restriction of characters is surjective.
Set $y=ww_0^{-1}$, where $w\in [W_\Theta \backslash W_0 / W_\Omega]$. 

If ${}^w M_\Omega \subseteq M_\Theta$, then the parabolic subgroup $P_\Theta\cap{}^wM_\Omega$ of ${}^wM_\Omega$ is equal to ${}^wM_\Omega$.  The containment ${}^w M_\Omega \subseteq M_\Theta$ occurs in two cases:
\begin{enumerate}
	\item[(A)]  when ${}^w M_\Omega = M_\Theta$, which occurs if and only if $w\in [W_\Theta \backslash W_0 / W_\Omega] \cap W(\Theta, \Omega)$, where $W(\Theta, \Omega) = \{ w \in W_0 : w\Omega = \Theta \}$, and
	\item[(B)] when $w\in [W_\Theta \backslash W_0/ W_\Omega]$ is such that ${}^wM_\Omega \subsetneq M_\Theta$ is a proper Levi subgroup of $M_\Theta$.
\end{enumerate}
 
In both cases (A) and (B),  $\mathcal F_\Theta^y(\tau)$ is not $M_\Theta^{\theta}$-distinguished by \Cref{no-dist-unit-exp} and the unitary exponents $\chi_{\Theta,y}$ of $\mathcal F_\Theta^y(\tau)$ do not contribute to $\Exp_{S_\Theta}(\pi_{N_\Theta}, r_{P_\Theta} \lambda)$ by \Cref{non-dist-gen-eig-sp} (\textit{cf.}~\eqref{relative-exponent}).

Otherwise, if ${}^w M_\Omega \nsubseteq M_\Theta$, we have that $P_\Theta\cap {}^w M_\Omega$ is a proper parabolic subgroup of ${}^w M_\Omega$. 
By \Cref{prop-reduce-Cas-ind}, we have that 
\begin{align*}
|\chi(s)|_F < 1, \ \text{for all} \ \chi \in \Exp_{S_\Theta}(\mathcal F_\Theta^y(\tau)), \ \text{and all} \ s \in S_\Theta^- \setminus S_\Theta^1 S_{\Delta_0}.
\end{align*}
In particular, \eqref{rel-casselman} holds for all exponents $\chi \in \Exp_{S_\Theta}(\pi_{N_\Theta}, r_{P_\Theta} \lambda)$ relative to $ \lambda$ along all maximal $\Delta_0$-standard $\theta$-split parabolic subgroups $P_\Theta$.  
By \Cref{rel-casselman-crit}, we conclude that $\pi$ is $(H,\lambda)$-relatively square integrable. 
\end{proof}

\begin{obs}\label{obs-isom-type-RDS}
	A representation $\pi$ of $G$ is $H$-distinguished if and only if $\pi$ is ${}^gH$-distinguished; in particular, the property of distinction only depends on the $G$-conjugacy class of $H$ (or the $G$-equivalence class of $\theta$, \textit{cf.}~\Cref{orbit-dist}).
	Thus, taking into account \Cref{ind-subgp-prop} and \Cref{pairwise-inequiv}, we may rephrase \Cref{main-RDS-thm} as follows:
	\begin{enumerate}
	\item Assume that $n$ is even. Let $(m_1, \ldots, m_k)$ be a partition of $n$ such that each $m_i$ is even.  Let $\tau_1, \ldots, \tau_k$ be pairwise inequivalent $H_{m_i}$-distinguished discrete series representations of $G_{m_i}$.  The parabolically induced representation $\tau_1 \times \ldots \times \tau_k$ is an irreducible $H$-distinguished relative discrete series representation of $G$.
	\item If $n$ is odd, then we must be in the Galois case. Let $(m_1, \ldots, m_k)$ be a partition of $n$ such that exactly one $m_l$ is odd, and all other $m_i$ are even.  Let $\tau_1, \ldots, \tau_k$ be pairwise inequivalent $\GL_{m_i}(F)$-distinguished discrete series representations of $\GL_{m_i}(E)$.  The parabolically induced representation $\tau_1 \times \ldots \times \tau_k$ is an irreducible $\GL_n(F)$-distinguished relative discrete series representation of $\GL_n(E)$.
	\end{enumerate}
\end{obs}

\begin{cor}\label{cor-rds-not-ds}
Let $\pi = \iota_Q^G\tau$ be as in \Cref{main-RDS-thm}.  
The representation $\pi$ is a relative discrete series representation that does not lie in the discrete spectrum of $G$.
\end{cor}

\begin{proof}
By \Cref{main-RDS-thm}, $\pi$ is irreducible and $H$-relatively square integrable; therefore, $\pi$ is a relative discrete series. Since $\pi = \iota_Q^G\tau$, where $Q$ is proper in $G$, it follows from the work of Zelevinsky \cite{zelevinsky1980} that $\pi$ does not occur in the discrete spectrum of $G$.
\end{proof}

\begin{rmk}\label{rmk-exhaustion}
At present, the author does not know if the construction outlined in \Cref{main-RDS-thm} exhausts all non-discrete relative discrete series in the linear and Galois cases.  In order to show that a representation is not $(H,\lambda)$-relatively square integrable, it is necessary to show that $r_P\lambda$ is non-vanishing on the generalized eigenspace corresponding to an exponent $\chi \in \Exp_{S_M}(\pi_N,r_P\lambda)$ that fails the condition \eqref{rel-casselman}.  The non-vanishing of $r_P\lambda$ is obscured by the nature of the construction of the form via Casselman's Canonical Lifting.  Due to this lack of precise information, we cannot exclude the possibility that certain representations are RDS.  For instance, it may be possible to relax the regularity condition imposed in \Cref{main-RDS-thm}, which is essential in the proof of \Cref{no-dist-unit-exp}.  At this time, the author does not have a method to remove the assumption of regularity from \Cref{no-dist-unit-exp}, due to a lack of information regarding the support of the $r_P\lambda$.
\end{rmk}

By \cite[Theorem 6.2]{kato--takano2008} and the proof of \Cref{main-RDS-thm}, 
one may obtain the following.

\begin{cor}\label{cor-sc-main-thm}
Let $Q = LU$ be as in \Cref{main-RDS-thm}.
If $\tau$ is a regular $L^\theta$-distinguished supercuspidal representation of $L$, then $\pi = \iota_Q^G\tau$ is $H$-relatively supercuspidal.
\end{cor}

\begin{rmk}
Note that \Cref{cor-sc-main-thm} can be obtained by more direct methods; see, for instance, the work of Murnaghan \cite{murnaghan2016-pp} for such results in a more general setting.
\end{rmk}

\section{Distinguished discrete series: Known results and inducing data}\label{sec-inducing-data}
In this section, we survey the known results on distinguished discrete series representations in the linear and Galois cases.  Our ultimate goal is to prove \Cref{prop-infinite-reg-dist-non-sc-ds} and thus \Cref{cor-infinite-family}.
First, we note that, in the linear case, multiplicity-one is due to Jacquet and Rallis \cite{jacquet--rallis1996}.
In the Galois case, multiplicity-one is due to Flicker \cite{flicker1991}.
\begin{prop}[Jacquet--Rallis, Flicker]\label{prop-mult-one}
Let $\pi$ be irreducible admissible representation of $G$.  If $\pi$ is $H$-distinguished, then $\Hom_H(\pi,1)$ is one-dimensional.
\end{prop}

Jacquet and Rallis \cite{jacquet--rallis1996} also prove the next proposition.

\begin{prop}[Jacquet--Rallis]\label{lin-dist-property}
Let $M$ be a maximal Levi subgroup of $\GL_m(F)$, where $m\geq 2$.
Let $\pi$ be an irreducible admissible representation of $\GL_m(F)$.
If $\pi$ is $M$-distinguished, then $\pi$ is equivalent to its contragredient $\widetilde \pi$.
\end{prop}

Flicker \cite{flicker1991} uses the methods of \cite{gelfand--kazhdan1975} to prove the following result.

\begin{prop}[Flicker]\label{Gal-dist-property}
Let $\pi$ be an irreducible admissible representation of $\GL_m(E)$, where $m\geq 2$. If $\pi$ is $\GL_m(F)$-distinguished, then $\pi \cong {}^{\theta}\widetilde\pi$.
\end{prop}

\subsection{Distinguished discrete series in the linear case}
In this subsection, unless otherwise noted, we let $G = \GL_n(F)$, where $n\geq 2$ is even, and let $H=\GL_{n/2}(F) \times \GL_{n/2}(F)$.

\begin{rmk}\label{rmk-lin-Shalika}
It is known that an irreducible square integrable representation $\pi$ of $G$ is $H$-distinguished if and only if $\pi$ admits a Shalika model. 
It was shown by Jacquet and Rallis \cite{jacquet--rallis1996} that if $\pi$ is an irreducible admissible representation of $G$ that admits a Shalika model, then $\pi$ is $H$-distinguished.  
For irreducible supercuspidal representations, the converse appears as \cite[Theorem 5.5]{jiang--nien--qin2008}.  
Independently, Sakellaridis--Venkatesh and Matringe proved the converse result for relatively integrable and relatively square integrable representations by the technique of ``unfolding" \cite[Example 9.5.2]{sakellaridis--venkatesh2017}, \cite[Theorem 5.1]{matringe2014a}. 
In fact, Sakellaridis--Venkatesh prove that there is an equivariant unitary isomorphism between $L^2(H\backslash G)$ and $L^2(S\backslash G)$, where $S$ is the Shalika subgroup.
Several analogous global results appear in \cite{gan--takeda2010}.
\end{rmk}

Let $\pi$ be a discrete series representation of $\GL_m(F)$, $m \geq 2$.  Denote by $L(s,\pi\times\pi)$ the local Rankin--Selberg convolution $L$-function.
It is well known that $L(s,\pi\times\pi)$ has a simple pole at $s=0$ if and only if $\pi$ is self-contragredient \cite{jacquet--piatetski-shaprio--shalika1983}.
By \cite[Lemma 3.6]{shahidi1992}, we have a local identity 
\begin{align}\label{eq-RS-L-fcn-factors}
L(s,\pi\times \pi) = L(s,\pi, \wedge^2) L(s,\pi, \Sym^2),
\end{align}
where $L(s,\pi, \wedge^2)$, respectively $L(s,\pi, \Sym^2)$, denotes the exterior square, respectively symmetric square, $L$-function of $\pi$ defined via the Local Langlands Correspondence (LLC).
It is also well known, see \cite{bump2004, kewat--raghunathan2012} for instance, that $L(s,\pi, \wedge^2)$ cannot have a pole when $m$ is odd.

\begin{thm}[{\cite[Proposition 6.1]{matringe2014a}}]\label{thm-matringe-ext-square}
Suppose that $\pi$ is a square integrable representation of $G$, then $\pi$ is $H$-distinguished
if and only if the exterior square $L$-function $L(s,\pi,\wedge^2)$ has a pole at $s=0$.
\end{thm}

\begin{rmk}
It is now known that for all discrete series, and when $n$ is even all irreducible generic representations, the Jacquet--Shalika  and Langlands--Shahidi local exterior square $L$-functions agree with with the exterior square $L$-function defined via the LLC \cite[Theorem 4.3 in $\S$4.2]{henniart2010}, \cite[Theorems 1.1 and 1.2]{kewat--raghunathan2012}.
\end{rmk}

Let $\rho$ be an irreducible unitary supercuspidal representation of $\GL_r(F)$, $r\geq 1$.  For an integer $k\geq 2$, write $\st(k,\rho)$ for the unique irreducible (unitary) quotient of the parabolically induced representation $\nu^{\frac{1-k}{2}} \rho \times \nu^{\frac{3-k}{2}} \rho \times \ldots \times \nu^{\frac{k-1}{2}} \rho$ of $\GL_{kr}(F)$ (\textit{cf.}~\cite[Proposition 2.10, $\S$9.1]{zelevinsky1980}), where $\nu(g) = |\det(g)|_F$, for any $g \in \GL_r(F)$. 
The representations $\st(k,\rho)$ are often called {generalized Steinberg representations} and they are exactly the nonsupercuspidal discrete series representations of $\GL_{kr}(F)$ \cite[Theorem 9.3]{zelevinsky1980}. 
The usual Steinberg representation $\st_n$ of $\GL_n(F)$ is obtained as $\st(n,1)$.
Note that $\st({k_1},\rho_1)$ is equivalent to $\st({k_2},\rho_2)$ if and only if $k_1 = k_2$ and $\rho_1$ is equivalent to $\rho_2$ \cite[Theorem 9.7(b)]{zelevinsky1980}.  

\begin{thm}[{\cite[Theorem 6.1]{matringe2014a}}]\label{thm-matringe-lin}
Suppose that $n= kr$ is even.  Let $\rho$ be an irreducible supercuspidal representation of $\GL_r(F)$.  
Let $\pi = \st(k,\rho)$ be a generalized Steinberg representation of $G$.
\begin{enumerate}
\item If $k$ is odd, then $r$ must be even, and $\pi$ is $H$-distinguished
if and only if $L(s, \rho, \wedge^2)$ has a pole at $s=0$ if and only if $\rho$ is $\GL_{r/2}(F)\times \GL_{r/2}(F)$-distinguished 
\item If $k$ is even, then $\pi$ is $H$-distinguished
if and only if $L(s, \rho, \Sym^2)$ has a pole at $s=0$. \label{thm-matringe-lin-even}
\end{enumerate}
\end{thm}

The following is a corollary of \cite[Proposition 10.1]{murnaghan2011a}  and \cite[Theorem 1.3]{Hakim--Murnaghan2002}.
\begin{thm}[Murnaghan, Hakim--Murnaghan]\label{thm-lin-HM}
For any even integer $n \geq 2$, there exist 
\begin{enumerate}
\item infinitely many equivalence classes of $H$-distinguished irreducible (unitary) tame supercuspidal representations of $G$, and \label{HM-dist}
\item infinitely many equivalence classes of self-contragredient irreducible (unitary) tame supercuspidal representations of $G$ that are not $H$-distinguished. \label{HM-not-dist}
\end{enumerate}
\end{thm}

Finally, we note the following.
\begin{prop}\label{prop-lin-even-dist-ds} 
For any even integer $n \geq 2$, there are infinitely many equivalence classes of $H$-distinguished discrete series representations of $G$.
Moreover,
\begin{enumerate}
\item if $n=2$, there are exactly four $H$-distinguished twists of the Steinberg representation $\st_2$ of $G$; \label{prop-lin-even-dist-ds--2}
\item if $n=4$, there are exactly four $H$-distinguished twists of the Steinberg representation $\st_4$ of $G$, and there are infinitely many equivalence classes of $H$-distinguished generalized Steinberg representations of $G$ of the form $\st(2,\rho)$.   \label{prop-lin-even-dist-ds--4} 
\item  if $n\geq 6$, there are infinitely many equivalence classes of $H$-distinguished nonsupercuspidal discrete series representations of $G$. \label{prop-lin-even-dist-ds--geq6}
\end{enumerate}
\end{prop}

\begin{proof}
The main statement follows from \Cref{thm-lin-HM}.
\begin{enumerate}
\item The Steinberg representation $\st_2$ has trivial central character and so it is $H$-distinguished \cite{prasad1993}.  
A twist $\chi \otimes \st_2$ of $\st_2$ by a quasi-character $\chi$ of $F^\times$ has trivial central character if and only if $\chi$ is trivial on $(F^\times)^2$. 
In particular, $\chi$ must be a quadratic (unitary) character and, since $F$ has odd residual characteristic, there are four distinct such characters.
\item By \cite[Corollary 8.5(\rm{ii})]{gan--takeda2010}, a twisted Steinberg representation $\chi\otimes\st_4$ admits a Shalika model if and only if $\chi$ is trivial on $(F^\times)^2$.
In this case, $\chi\otimes\st_4$ is $H$-distinguished (\textit{cf.}~\Cref{rmk-lin-Shalika}).

By \Cref{thm-lin-HM}\eqref{HM-not-dist}, there exist infinitely many classes of self-contragredient supercuspidal representations $\rho$ of $G_2$ that are not $H_2$-distinguished.  In particular, given such a $\rho \cong \wt\rho$, the Rankin--Selberg $L$-fucntion $L(s,\rho\times\rho)$ has a pole at $s=0$ \cite{jacquet--piatetski-shaprio--shalika1983}; however, by \Cref{thm-matringe-ext-square}, $L(s,\rho,\wedge^2)$ does not have a pole at $s=0$.  It follows from \eqref{eq-RS-L-fcn-factors} that $L(s,\rho, \Sym^2)$ has a pole at $s=0$.
The claim follows from \Cref{thm-matringe-lin}\eqref{thm-matringe-lin-even}. 
%
\item  The last statement is an immediate consequence of \Cref{thm-matringe-lin,thm-lin-HM}.
\end{enumerate}
\end{proof}
\subsection{Distinguished discrete series in the Galois case}
In this subsection, unless otherwise noted, let $G = R_{E/F} \GL_n(F)$, where $n\geq 2$. 
We identify $G$ with $\GL_n(E)$. 
Let $H = \GL_n(F)$ be the subgroup of Galois fixed points in $G$.
Let $\eta: E^\times \rightarrow \C^\times$ be an extension to $E^\times$ of the character $\eta_{E/F}:F^\times \rightarrow \C$ associated to $E/F$ by local class field theory. 
The following result is due to Anandavardhanan--Rajan \cite[Section 4.4]{anandavardhanan--rajan2005}, and also appears as \cite[Theorem 1.3]{anandavardhanan2008} and \cite[Corollary 4.2]{matringe2009a}.

\begin{thm}[Anandavardhanan--Rajan]\label{thm-matringe-Gal}
Let $\rho$ be an irreducible supercuspidal representation of $\GL_r(E)$, then the generalized Steinberg representation $\pi = \st(k,\rho)$ of $\GL_{kr}(E)$ is $\GL_{kr}(F)$-distinguished if and only if $\rho$ is $(\GL_r(F),\eta_{E/F}^{k-1})$-distinguished.
\end{thm}

The next result is due to Prasad for $n=2$ and Anandavardhanan--Rajan for $n\geq 3$.
\begin{thm}[{\cite{prasad1992}},{\cite[Theorem 1.5]{anandavardhanan--rajan2005}}]\label{thm-Gal-st}
Let $\chi$ be a quasi-character of $F^\times$ and identify $\chi$ with the quasi-character $\chi \circ \det$ of $H$.
The Steinberg representation $\st_n$ of $G$ is $(H,\chi)$-distinguished if and only if:
\begin{enumerate}
\item $n$ is odd and $\chi = 1$, or \label{thm-Gal-st-odd}
\item $n$ is even and $\chi = \eta_{E/F}$. \label{thm-Gal-st-even}
\end{enumerate}
\end{thm}
\begin{cor}\label{cor-Gal-st-2}
If $n=2$, then the twist $\eta\otimes \st_2$ of the Steinberg representation $\st_2$ of $G$ is $H$-distinguished.
\end{cor}

The following is a corollary of \cite[Proposition 10.1]{murnaghan2011a} and \cite[Theorem 1.1]{Hakim--Murnaghan2002}.
\begin{thm}[Hakim--Murnaghan]\label{thm-Gal-HM}
There are infinitely many equivalence classes of
\begin{enumerate}
\item irreducible (unitary) supercuspidal representations of $G$ that are $H$-distinguished.  
\item irreducible (unitary) supercuspidal representations of $G$ that are $(H, \eta_{E/F})$-distinguished. 
\end{enumerate}
\end{thm}

The next result is an immediate consequence of \Cref{thm-matringe-Gal,thm-Gal-st,thm-Gal-HM}.
\begin{cor}\label{prop-Gal-infinite-non-sc-dist-ds}
Let $G = \GL_n(E)$ and let $H = \GL_n(F)$.
\begin{enumerate}
\item If $n \geq 4$ is not equal to an odd prime, then there are infinitely many equivalence classes of $H$-distinguished nonsupercuspidal discrete series representations of $G$.
\item If $n$ is equal to an odd prime, then the Steinberg representation $\st_n$ of $G$ is a nonsupercuspidal $H$-distinguished discrete series.
\end{enumerate}
\end{cor}

\begin{proof}
Assume that $n \geq 4$ is not an odd prime.  Then $n = kr$ for two integers $k,r \geq 2$.  Note that $\eta_{E/F}$ is a quadratic character; in particular, if $k$ is even, then $\eta_{E/F}^{k-1} = \eta_{E/F}$ and if $k$ is odd, then $\eta_{E/F}^{k-1} = 1$.  By \Cref{thm-Gal-HM}, there are infinitely many equivalence classes of irreducible supercuspidal representations $\rho$ of $G_r$ that are $(G_r, \eta_{E/F}^{k-1})$-distinguished.  By \Cref{thm-matringe-Gal} and  \cite[Theorem 9.7(b)]{zelevinsky1980}, there are infinitely many equivalence classes of generalized Steinberg representations of $G$ of the form $\st(k,\rho)$ and that are $H$-distinguished.
Of course, the generalized Steinberg representations $\st(k,\rho)$ are nonsupercuspidal discrete series representations.
The second statement follows from \Cref{thm-Gal-st}\eqref{thm-Gal-st-odd}.
\end{proof}
\subsection{The inducing representations in \Cref{main-RDS-thm}}
For the remainder of the paper, fix a proper $\Delta^{ell}$-standard $\theta$-stable parabolic subgroup $Q = Q_{\Omega^{ell}}$, for some proper subset $\Omega^{ell} \subset {\Delta}^{ell}$ containing $\Delta^{ell}_{\min}$.  
As in \Cref{ind-subgp-prop}, the subgroup $Q$ admits a standard-$\theta$-elliptic Levi subgroup $L = L_{\Omega^{ell}}$ and unipotent radical $U= U_{\Omega^{ell}}$.
The next lemma is straightforward to verify by using the description of $L^\theta$ given in (the proof of) \Cref{ind-subgp-prop} and \Cref{orbit-dist}.  The multiplicity-one statement follows from \Cref{prop-mult-one}.

\begin{lem}\label{lem-L-theta-dist}
Let $L \cong \prod_{i=1}^k G_{m_i}$ be a standard-$\theta$-elliptic Levi subgroup of $G$.  Let $\tau \cong \bigotimes_{i=1}^k \tau_i$ be an irreducible admissible representation of $L$ where each $\tau_i$ is an irreducible admissible representation of $G_{m_i}$, $1\leq i \leq k$.
\begin{enumerate}
\item Then $\tau$ is $L^\theta$-distinguished if and only if $\tau_i$ is $H_{m_i}$-distinguished for all $1\leq i \leq k$.
\item If $\tau$ is $L^\theta$-distinguished, then $\Hom_{L^\theta}(\tau,1)$ is one-dimensional.
\end{enumerate}
\end{lem}

\begin{prop}\label{prop-infinite-reg-dist-non-sc-ds}
Let $L$ be a standard-$\theta$-elliptic Levi subgroup of $G$.
There exist infinitely many equivalence classes of regular non-supercuspidal $L^\theta$-distinguished discrete series representations of $L$.
\end{prop}

\begin{proof}
By assumption $n\geq 4$ and $n$ is always taken to be even in the linear case.
We have that $L$ is isomorphic to a product $\prod_{i=1}^k G_{m_i}$ of smaller general linear groups, for some partition $(m_1,\ldots, m_k)$ of $n$.
Let $\tau_i$ be an irreducible admissible representation of $G_{m_i}$, $1 \leq i \leq k$.
By \Cref{pairwise-inequiv}, the representation $\tau_1 \otimes \ldots \otimes \tau_k$ of $L$ is regular if and only if $\tau_i \ncong \tau_j$ for all $1\leq i \neq j \leq k$.
Moreover, $\tau_1 \otimes \ldots \otimes \tau_k$  is supercuspidal (square integrable) if and only if every $\tau_i$ is supercuspidal (square integrable).
It is sufficient to prove that for any relevant partition of $n$ (\textit{cf.} \Cref{ind-subgp-prop}), there exist pairwise inequivalent $H_{m_i}$-distinguished discrete series representations $\tau_i$, such that at least one $\tau_i$ is not supercuspidal.
 
In the linear case, by \Cref{ind-subgp-prop},  each $m_i \geq 2$ is even.  By \Cref{thm-lin-HM}, there are infinitely many equivalence classes of $H_{m_i}$-distinguished irreducible supercuspidal representations of $G_{m_i}$.
By \Cref{prop-lin-even-dist-ds}, there exists at least one non-supercuspidal $H_{m_i}$-distinguished discrete series representation of $G_{m_i}$ (infinitely many when $m_i \geq 4$).  It follows from \Cref{lem-L-theta-dist} that there exist infinitely many equivalence classes of regular non-supercuspidal $L^\theta$-distinguished discrete series representations of $L$.

In the Galois case, by \Cref{ind-subgp-prop}, at most one $m_i$  is odd. Without loss of generality, assume that $m_k$ is odd.  
By \Cref{thm-Gal-st}\eqref{thm-Gal-st-odd}, the Steinberg representation $\st_{m_k}$ of $G_{m_k}$ is a non-supercuspidal $H_{m_k}$-distinguished discrete series. 
By \Cref{thm-Gal-HM}, there are infinitely many equivalence classes of $H_{m_i}$-distinguished irreducible supercuspidal representations of $G_{m_i}$.  
By \Cref{cor-Gal-st-2} and \Cref{prop-Gal-infinite-non-sc-dist-ds}, there exists at least one non-supercuspidal $H_{m_i}$-distinguished discrete series representation of $G_{m_i}$ (infinitely many when $m_i\geq 4$ is not an odd prime).
It follows from \Cref{lem-L-theta-dist} that there exist infinitely many equivalence classes of regular non-supercuspidal $L^\theta$-distinguished discrete series representations of $L$. 
\end{proof}

\begin{cor}\label{cor-infinite-family}
There are infinitely many equivalence classes of $H$-distinguished relative discrete series representations of $G$ of the form constructed in \Cref{main-RDS-thm}. In particular, there are infinitely many classes of such representations where the discrete series $\tau$ is not supercuspidal.
\end{cor}
\begin{proof}
This is immediate from \Cref{prop-infinite-reg-dist-non-sc-ds} and \cite[Theorem 9.7(b)]{zelevinsky1980}.
\end{proof}
\section{Computation of exponents and distinction of Jacquet modules}\label{sec-exp-dist-pi-N}
We work under the hypotheses of Theorem \ref{main-RDS-thm} and use the notation of its proof.
In order to discuss Casselman's Criterion for the inducing data of $\pi=\iota_Q^G\tau$ we use the following notation.
If $\Theta_1 \subset \Theta_2 \subset \Delta_0$, then we define
\begin{align*}
A_{\Theta_1}^- = \{ a \in A_{\Theta_1} : |\alpha(a)|\leq 1,\ \text{for all} \ \alpha \in \Delta_0\setminus {\Theta_1}\}
\end{align*}
and
\begin{align*}
A_{\Theta_1}^{-\Theta_2} = \{ a \in A_{\Theta_1} :  |\beta(a)|\leq 1,\ \text{for all} \ \beta \in \Theta_2 \setminus {\Theta_1}\}.
\end{align*}
The set $A_{\Theta_1}^-$ is the dominant part of $A_{\Theta_1}$ in $G$, while $A_{\Theta_1}^{-\Theta_2}$ is the dominant part of $A_{\Theta_1}$ in $M_{\Theta_2}$.

In both the linear and Galois cases, we have that $Q = w_0 P_\Omega w_0^{-1}$ is a Weyl group conjugate of a $\Delta_0$-standard parabolic subgroup $P_\Omega$, where $\Omega^{ell} = w_0 \Omega$. We also have that $L = w_0 M_\Omega w_0^{-1}$.
If $\tau_0$ is a representation of $M_\Omega$, then $\tau = {}^{w_0}\tau_0$ is a representation of $L$.  
Let $P_\Theta$ be a $\Delta_0$-standard maximal $\theta$-split parabolic subgroup corresponding to a maximal proper $\theta$-split subset $\Theta \subset \Delta_0$.  Below $y = ww_0^{-1}$ is a ``nice" representative of a double-coset in $P_\Theta \backslash G / Q$, where $w\in [W_\Theta \backslash W_0/ W_\Omega]$ (\textit{cf.}~\Cref{lem-nice-reps}).

\begin{lem}\label{red-to-Cas-ind-data} 
The exponents of $\mathcal F_\Theta^y(\tau)$ are the restriction to $A_\Theta$ of the exponents of ${}^w \tau_0$ along the parabolic subgroup $P_\Theta\cap {}^w M_\Omega$ of ${}^w M_\Omega$. 
If $\tau = {}^{w_0} \tau_0$ is a discrete series representation of $L$, and the parabolic subgroup $P_\Theta\cap {}^w M_\Omega$ of ${}^w M_\Omega$ is proper,  then for any exponent $\chi \in \Exp_{A_\Theta}(\mathcal F_\Theta^y(\tau))$ the inequality $|\chi(a)|_F < 1$ is satisfied for every $a \in A_{\Theta \cap w\Omega}^{-w\Omega} \setminus  A_{\Theta \cap w\Omega}^1 A_{w\Omega}$.
\end{lem}

\begin{proof}
This is a special case of  \Cref{red-to-ind-exp} and the usual Casselman's Criterion (\textit{cf.}~\cite[Theorem 6.5.1]{Casselman-book}) applied to the discrete series representation ${}^w\tau_0$ of ${}^w M_\Omega$.
\end{proof}

\begin{lem}\label{unitary-exponents} 
Assume that $\tau$ is a regular unitary irreducible admissible representation of $L$.  
If $y= w w_0^{-1}$, where $w\in [W_\Theta \backslash W_0/ W_\Omega]$, is such that ${}^w M_\Omega \subset P_\Theta$, then $\mathcal F_\Theta^y(\tau)$ is irreducible and the central character $\chi_{\Theta,y}$ of $\mathcal F_\Theta^y(\tau)$ is unitary.
\end{lem}

\begin{proof}
If ${}^w M_\Omega \subset P_\Theta$, then $P_\Theta\cap {}^w M_\Omega = {}^w M_\Omega$, $N_\Theta \cap {}^w M_\Omega = \{e\}$, and ${}^wM_\Omega \subset M_\Theta$.
It follows that the representation $({}^w \tau_0)_{N_\Theta\cap {}^w M_\Omega}$ is equal to ${}^w \tau_0$ and it is irreducible and unitary. 
Moreover, since $\tau$ is a regular representation of $L$, it follows that ${}^w \tau_0$ is regular as a representation of ${}^w M_\Omega$ regarded as a Levi subgroup of $M_\Theta$.
By \cite[Theorem 6.6.1]{Casselman-book}, the representation $\mathcal F_\Theta^y(\tau)$ is irreducible and unitary.  
By the irreducibility of $\mathcal F_\Theta^y(\tau)$, the only exponent is its central character $\chi_{\Theta,y}$.
Since $\mathcal F_\Theta^y(\tau)$ is unitary, the character $\chi_{\Theta,y}$ of $A_\Theta$ is unitary.
\end{proof}

\begin{rmk}\label{rmk-two-cases-lin-Gal}
Recall that $W(\Theta, \Omega) = \{ w \in W_0 : w\Omega = \Theta \}$.
We find ourselves in the situation of \Cref{unitary-exponents} in two cases:
\begin{itemize}
	\item Case (A):   $w\in [W_\Theta \backslash W_0 / W_\Omega] \cap W(\Theta, \Omega)$,  if and only if ${}^w M_\Omega = M_\Theta$,
	\item Case (B):  $w\in [W_\Theta \backslash W_0/ W_\Omega]$ is such that ${}^wM_\Omega \subsetneq M_\Theta$ is a proper Levi subgroup of $M_\Theta$.
\end{itemize}
\end{rmk}

In order to apply the Relative Casselman's Criterion \ref{rel-casselman-crit}, using \Cref{red-to-Cas-ind-data}, we need the following technical fact.
%
%
\begin{lem}\label{technical-torus-nbhd}
Let $\Omega$ be a proper subset of ${\Delta_0}$ such that $\Omega^{ell} = w_0\Omega$ contains $\Delta^{ell}_{\min}$. Let ${\Theta}$ be a maximal ${\theta}$-split subset of ${\Delta_0}$.  Let $w\in [W_\Theta \backslash W_0 / W_\Omega]$ such that $M_{\Theta\cap w\Omega} = M_\Theta \cap wM_\Omega w^{-1}$ is a proper Levi subgroup of $M_{w\Omega} = w M_\Omega w^{-1}$. Then we have the containment:
\begin{align*}
S_\Theta^- \setminus S_\Theta^1 S_{\Delta_0} \subset A_{\Theta\cap w\Omega}^{-w\Omega} \setminus A_{\Theta\cap w\Omega}^1A_{w\Omega}.
\end{align*}
\end{lem}

\begin{proof}
Recall that for any $F$-torus $\mathbf A$ we write $A^1$ for the $\of$-points of $\mathbf A$.
First, $S_\Theta$ is contained in $A_\Theta$ and since $\Theta\cap w\Omega$ is a subset of $\Theta$, we have that 
\begin{align*}
\mathbf A_\Theta & = \left( \bigcap_{\alpha\in{\Theta}} \ker \alpha \right)^\circ 
 \subset  \left( \bigcap_{\alpha\in\Theta\cap w\Omega} \ker \alpha \right)^\circ 
 = \mathbf A_{\Theta\cap w\Omega}.
\end{align*}
At the level of $F$-points, we have $A_\Theta \subset A_{\Theta\cap w\Omega}$, and similarly for the integer points
$A_\Theta^1 \subset A_{\Theta\cap w\Omega}^1$.  It follows that $S_\Theta \subset A_{\Theta\cap w\Omega},$ and $S_\Theta^1 \subset A_{\Theta\cap w\Omega}^1$.  Also, we have $S_{\Delta_0} \subset A_{\Delta_0} \subset A_\Omega$, and since $S_{\Delta_0}=S_G$ is central in $G$, we have $S_{\Delta_0} = wS_{\Delta_0} w^{-1} \subset wA_\Omega w^{-1} = A_{w\Omega}$.  
We now observe that $A_\Theta^- \subset A_{\Theta\cap w\Omega}^{-w\Omega}$ and in particular that $S_\Theta^- \subset A_{\Theta\cap w\Omega}^{-w\Omega}$; it is clear that $S_\Theta^- \subset A_\Theta^-$. Note that $w\Omega$ is a base for the root system of $M_{w\Omega}$ relative to the maximal $F$-split torus ${A_0}$. Suppose that $a \in A_\Theta^-$, then $|\alpha(a)| \leq 1$, for all $\alpha \in {\Delta_0} \setminus {\Theta}$. 
Moreover, since $a\in A_\Theta$ we have that $|\alpha(a)| = 1$, for $\alpha \in {\Theta}$ as well.  Let $\beta \in w\Omega \setminus (\Theta\cap w\Omega)$, then $\beta =w\alpha$ for some $\alpha \in \Omega$.  Since $w\in [W_\Theta \backslash W_0 / W_\Omega]$, we have that $\beta = w\alpha \in \Phi_0^+$.  Write $\beta = \sum_{\epsilon \in {\Delta_0}} c_\epsilon \cdot \epsilon$,
where $c_\epsilon \geq 0$, $c_\epsilon \in \Z$.  Then we have that
\begin{align*}
|\beta(a)| = \left\vert \prod_{\epsilon \in {\Delta_0}} \epsilon(a)^{c_\epsilon} \right \vert =  \prod_{\epsilon \in {\Delta_0}} |\epsilon(a)|^{c_\epsilon} \leq 1,
\end{align*}
since $|\epsilon(a)|\leq 1$, for all $\epsilon \in {\Delta_0}$, and $c_\epsilon \geq 0$.  In particular, $a \in A_{\Theta\cap w\Omega}^{-w\Omega}$.
Putting this together, we see that $S_\Theta^1 S_{\Delta_0} \subset S_\Theta^- \cap A_{\Theta\cap w\Omega}^1A_{w\Omega}$; therefore, to prove the desired result, it suffices to prove the opposite inclusion. 

It is at this point that we specialize to the two explicit cases. By assumption ${\Theta} = \Theta_k$, for some $1\leq k \leq \floor{\frac{n}{2}}$, as in \Cref{max-theta-split}.
Suppose that $s\in S_\Theta^- \cap A_{\Theta\cap w\Omega}^1A_{w\Omega}$.  We want to show that $s \in S_\Theta^1 S_{\Delta_0}$.  
Notice that $S_{\Delta_0} = \{ e\} $; therefore, it is sufficient to prove that $s\in S_\Theta^1$.  
By assumption, $s = tz$ where $t\in A_{\Theta\cap w\Omega}^1$ and $z \in A_{w\Omega}$.  Since $w\in [W_\Theta \backslash W_0 / W_\Omega]$, we have that $w\Omega \subset \Phi_0^+$; 
moreover, by the assumption that $M_{\Theta\cap w\Omega}$ is a proper Levi subgroup of $M_{w\Omega}$, 
we have that ${\Theta} \cap w\Omega \subsetneq w\Omega$ is a proper subset.  
It follows that $w\Omega$ cannot be contained in $\Phi_\Theta^+$.
Moreover, there exists $\alpha \in w\Omega \setminus ({\Theta} \cap w\Omega)$ such that $\alpha \in \Phi_0^+$ and $\alpha \notin \Phi_\Theta^+$.  
In the Galois case, there is a unique expression
$\alpha  = \sum_{j=1}^{n-1} c_j {}^\gamma (\epsilon_j -\epsilon_{j+1})$, where $c_j \in \Z$ and $c_j \geq 0$,
such that, since ${\Theta} = \Theta_k$ and $\alpha \notin \Phi_\Theta^+$, at least one of $c_k$ or $c_{n-k}$ is nonzero ($c_{n/2} \neq 0$, when $n$ even, $k=n/2$).  
In the linear case, $\gamma$ doesn't appear in the expression for $\alpha$.
First observe that
$\alpha(s) = \alpha(t)\alpha(z) = \alpha(t) \in \of^\times$,
since $z\in A_{w\Omega}$ and $t\in A_{\Theta\cap w\Omega}^1 = A_{\Theta\cap w\Omega}(\of)$.  
On the other hand,
in the Galois case, writing $s$ explicitly as $s = {}^\gamma s'$, where
\begin{align*}
s' & =
 \left\{ \begin{array}{ll} \diag(\underbrace{a,\ldots,a}_{k}, \underbrace{1\ldots,1}_{\floor{\frac{n}{2}}-2k},\underbrace{a^{-1},\ldots, a^{-1}}_{k}), & 1\leq k \leq \floor{\frac{n}{2}} \vspace{0.25cm}\\
\diag(\underbrace{a,\ldots,a}_{n/2}, \underbrace{a^{-1},\ldots, a^{-1}}_{n/2}), & n \ \text{even}, k=n/2
\end{array} \right.
\end{align*}
(with $s = s'$ and $n$ even in the linear case). Applying $\alpha$ to $s$,  we have
\begin{align*}
\alpha(s) & = \left(\sum_{j=1}^{n-1} c_j {}^\gamma (\epsilon_j -\epsilon_{j+1})\right) ({}^\gamma s') 
 = (\epsilon_k- \epsilon_{k+1}) (s')^{c_k} (\epsilon_{n-k}- \epsilon_{n-k+1}) (s')^{c_{n-k}} 
 = a^{c_k} a^{c_{n-k}} 
\end{align*}
So in the Galois case, $\alpha(s) = a^c \in \of^\times$, where $c = c_k + c_{n-k}$ (or $c=2 c_{\frac{n}{2}}$ when $n$ is even and $k = n/2$), and similarly in the linear case.
Then we have $|a|_F^c = 1$ for $c$ a positive integer 
so $|a|_F = 1$. In particular, we have that $a\in \of^\times$, and $s\in S_\Theta^1 = S_\Theta(\of)$, as desired.
\end{proof}

\begin{prop}\label{prop-reduce-Cas-ind}
If $y=ww_0^{-1} \in [W_\Theta \backslash W_0 / W_\Omega]\cdot w_0^{-1}$ is such that 
$P_\Theta\cap {}^w M_\Omega$ is a proper parabolic subgroup of ${}^w M_\Omega$, then the exponents $\chi \in \Exp_{S_\Theta}(\mathcal F_\Theta^y(\tau))$ of $\mathcal F_\Theta^y(\tau)$ satisfy
$$|\chi(s)|_F < 1,$$ for all $s \in S_\Theta^- \setminus S_\Theta^1 S_{\Delta_0}$.
\end{prop}

\begin{proof}
First, by \Cref{technical-torus-nbhd}, we have that $S_\Theta^- \setminus S_\Theta^1 S_{\Delta_0} \subset A_{\Theta\cap w\Omega}^{-w\Omega} \setminus A_{\Theta\cap w\Omega}^1A_{w\Omega}$.
By \Cref{closed-subgp-centre}, any exponent $\chi \in \Exp_{S_\Theta}(\mathcal F_\Theta^y(\tau))$ is the restriction to $S_\Theta$ of an exponent $\widehat\chi \in \Exp_{A_\Theta}(\mathcal F_\Theta^y(\tau))$; therefore, the result follows from \Cref{red-to-Cas-ind-data}.
\end{proof}

Finally, we study $M_\Theta^\theta$-distinction of the Jacquet module $\pi_{N_\Theta}$.
In preparation for this, we characterized the $\theta$-fixed points of the standard Levi $M_r=M_{\Theta_r}$ of the maximal $\theta$-split parabolic subgroups $P_r=P_{\Theta_r}$ in \Cref{prop-M-Theta-fixed-pts}.
The characterization is in terms of the groups $M_{(r,n-2r,r)}$ and $H_{(r,n-2r,r)}$.
First, we note \Cref{H-bullet-dist}, which characterizes $H_{(r,n-2r,r)}$-distinction in the Galois case. We omit the elementary verification (and the obvious modification when $n$ is even and $r=n/2$).
\begin{lem}\label{H-bullet-dist}
Assume that we are in the Galois case.  Fix an integer $1\leq r \leq \floor{\frac{n}{2}}$ and assume that $r \neq n/2$.
Let $\pi_1\otimes \pi_2 \otimes \pi_3$ be an irreducible admissible representation of $M_{(r,n-2r,r)}$. Then $\pi_1\otimes \pi_2 \otimes \pi_3$ is $H_{(r,n-2r,r)}$-distinguished if and only if $\pi_2$ is $H_{n-2r}$-distinguished and $\pi_3 \cong {}^{\theta_k} \widetilde\pi_1$.
\end{lem}

\begin{prop}\label{no-dist-unit-exp}
Let $\tau \cong \bigotimes_{i=1}^k \tau_i$ be an irreducible admissible regular representation of $L$.
Assume that $\tau$ is $L^{\theta}$-distinguished.
Let $P_\Theta$ be a maximal $\Delta_0$-standard parabolic subgroup corresponding to a maximal $\theta$-split subset $\Theta$ of $\Delta_0$.
Let $y = ww_0^{-1}$, where $w \in [W_\Theta \backslash W_0 / W_\Omega]$.
In both Case (A) and Case (B), $\mathcal F_\Theta^y(\tau)$ cannot be $M_\Theta^{\theta}$-distinguished.
\end{prop}

\begin{proof}
By assumption, we are in either Case (A) or Case (B) of \Cref{rmk-two-cases-lin-Gal}.  In particular, we have that $y = ww_0^{-1}$, where $w\in [W_\Theta \backslash W_0 / W_\Omega]$, such that
\begin{align*}
\mathcal F_\Theta^y(\tau) & = \iota_{M_\Theta \cap {}^yQ}^{M_\Theta} ({}^y \tau) = \iota_{M_\Theta \cap {}^w M_\Omega}^{M_\Theta} ({}^w \tau_0)
\end{align*}
and
\begin{align*}
\tau_0 = {}^{w_0^{-1}} \tau = {}^\gamma \left( \bigotimes_{i=1}^k \tau_i \right)
\end{align*}
 is the representation of $M_\Omega$ corresponding to the representation $\tau$ of $L = L_{\Omega^{ell}} = w_0 M_\Omega w_0^{-1}$. 
 Where $\gamma$ simply doesn't appear in the linear case.
By \Cref{lem-L-theta-dist}, $\tau$ is $L^\theta$-distinguished if and only if each $\tau_i$ is $H_{m_i}$-distinguished for all $1 \leq i \leq k$.  
By \Cref{lem-max-split-subsets}, $\Theta$ is equal to $\Theta_r$ for some $1 \leq r \leq \floor{\frac{n}{2}}$ and $M_\Theta = M_{\Theta_r} =M_r$. 
Without loss of generality $r < n/2$. 
By \Cref{max-theta-split}, in the linear case $M_r = M_{(r,n-2r,r)}$ and in the Galois case $M_r = {}^\gamma M_{(r,n-2r,r)}$. 
We again use the shorthand $M_\bullet = M_{(r,n-2r,r)}$ and $H_\bullet = H_{(r,n-2r,r)}$. 
The $\theta$-fixed point subgroup of $M_r$ is described in \Cref{prop-M-Theta-fixed-pts}.  In the linear case, we have that
\begin{align*}
M_r^\theta =H_\bullet = \{ \diag(A, B, \theta_r(A)) : A \in G_r, B \in H_{n-2r} \},
\end{align*}
while in the Galois case, we have
\begin{align*}
M_r^\theta = \gamma \gamma_\bullet^{-1} H_\bullet \gamma_\bullet \gamma^{-1},
\end{align*}
where $\gamma_\bullet = \gamma_{(r,n-2r,r)} = \diag(\gamma_r,\gamma_{n-2r},\gamma_r) \in M_{(r,n-2r,r)}$.
By \Cref{orbit-dist}, in the Galois case, $\mathcal F_\Theta^y(\tau)$ is $M_\Theta^{\theta}$-distinguished if and only if ${}^{\gamma^{-1}}\mathcal F_\Theta^y(\tau)$ is $H_\bullet$-distinguished.  

Without loss of generality, we complete the proof in the Galois case.  
To obtain the proof in the linear case replace the application of \Cref{Gal-dist-property} with \Cref{lin-dist-property}.

\textbf{Case (A).}
	Suppose that ${\Theta}$ and $\Omega$ are associate and $w\in [W_\Theta \backslash W_0 / W_\Omega] \cap W(\Theta, \Omega)$.  Then $M_\Theta = {}^w M_\Omega$ and $\mathcal F_\Theta^y(\tau) = {}^w\tau_0 = {}^{w\gamma}(\tau_1\otimes \tau_2 \otimes \tau_3)$, where $\tau_1,\tau_3$ are representations of $G_r$ and $\tau_2$ is a representation of $G_{n-2r}$.  By convention, $\gamma^{-1}w \gamma$ is a permutation matrix (\textit{cf.}~\Cref{Weyl-gp-conjugate}). 
	Moreover, ${}^{\gamma^{-1}}(\mathcal F_\Theta^y(\tau)) = {}^{\gamma^{-1}w \gamma}(\tau_1\otimes \tau_2 \otimes \tau_3)$ is equal to $\tau_{x(1)}\otimes \tau_{x(2)} \otimes \tau_{x(3)}$, for some compatible permutation $x$ of $\{1,2,3\}$.  
	 By  \Cref{H-bullet-dist}, ${}^{\gamma^{-1}}\mathcal F_\Theta^y(\tau)$ is $H_\bullet$-distinguished if and only if $\tau_{x(3)} \cong {}^{\theta_r}\widetilde{\tau_{x(1)}}$ and $\tau_{x(2)}$ is $H_{n-2r}$-distinguished.  By \Cref{Gal-dist-property}, each $\tau_i$ satisfies $\tau_i \cong {}^{\theta_{m_i}}\widetilde{\tau_i}$. However, since $\tau$ is regular, \Cref{pairwise-inequiv} implies that the $\tau_i$ are pairwise inequivalent. In particular, we have that
	\begin{align*}
	{}^{\theta_r}\widetilde{\tau_{x(1)}} \cong \tau_{x(1)} \ncong \tau_{x(3)};
	\end{align*}
	therefore, ${}^{\gamma^{-1}}(\mathcal F_\Theta^y(\tau))$ is not $H_\bullet$-distinguished and $\mathcal F_\Theta^y(\tau)$ is not $M_\Theta^{\theta}$-distinguished.
	
\textbf{Case (B).}
	Suppose that $w\in [W_\Theta \backslash W_0 / W_\Omega]$ is such that $wM_\Omega w^{-1} \subset M_\Theta$ is a proper Levi subgroup. 
	 In this case, $\mathcal F_\Theta^y(\tau) = \iota_{M_\Theta\cap{}^wP_\Omega}^{M_\Theta} {}^w \tau_0$ and
	  $\tau_0$ is a regular irreducible unitary representation. By \cite[Theorem 6.6.1]{Casselman-book}, the representation ${}^{\gamma^{-1}}(\mathcal F_\Theta^y(\tau))$ is an irreducible unitary representation of $M_\bullet$.  
	 Indeed, $M_\Theta = {}^\gamma M_\bullet$ and $M_\Omega = {}^\gamma M$, with $M_\bullet = M_{(r,n-2r,r)}$ and $M=M_{(m_1,\ldots,m_k)}$; 
	 moreover,  $w' = {}\gamma^{-1} w \gamma$ is an element of $[W_{M_\bullet} \backslash W / W_{M}]$ and conjugates $M$ into $M_\bullet$.  
	 Writing $P$ for $P_{(m_1,\ldots,m_k)}$, we have
	\begin{align*}
	{}^{\gamma^{-1}}(\mathcal F_\Theta^y(\tau)) & = {}^{\gamma^{-1}}\left(\iota_{M_\Theta\cap{}^wP_\Omega}^{M_\Theta} {}^w \tau_0\right) 
	 = {}^{\gamma^{-1}}\left(\iota_{{}^\gamma M_\bullet \cap{}^{w\gamma}P}^{{}^\gamma M_\bullet} {}^{w\gamma} (\tau_1\otimes\ldots \otimes \tau_k)\right) 
	 \cong \iota_{M_\bullet \cap{}^{w'}P}^{M_\bullet} {}^{w'} (\tau_1\otimes\ldots \otimes \tau_k)
	\end{align*}
	and this representation is isomorphic to $\pi_1 \otimes \pi_2 \otimes \pi_3$,
	where $\pi_1, \pi_2$ are irreducible admissible representations of $G_r$ and $\pi_2$ is an irreducible admissible representation of $G_{n-2r}$. 
	Since $w' \in [W_{M_\bullet} \backslash W / W_{M}]$, by \Cref{casselman1-3-3}, 
	the group $M_\bullet \cap{}^{w'}P$ is a parabolic subgroup of $M_\bullet$
	(a product of parabolic subgroups on each block of $M_\bullet$). 
	It follows that each of the $\pi_j$, $j=1,2,3$, are  irreducibly induced representations of the form $\tau_{a_1} \times \ldots \times \tau_{a_r}$, for some subset of the representations $\{\tau_1, \ldots, \tau_k\}$.  Again, by \Cref{H-bullet-dist} ${}^{\gamma^{-1}}(\mathcal F_\Theta^y(\tau))$ is $H_\bullet$-distinguished if and only if $\pi_2$ is $G_{n-2r}$-distinguished and $\pi_3 \cong {}^{\theta_r} \widetilde{\pi_1}$.  Suppose that $\pi_1 = \tau_{a_1} \times \ldots \times \tau_{a_{l}}$, and $\pi_3 = \tau_{b_1} \times \ldots \times \tau_{b_{s}}$, then we have that
	\begin{align*}
	\widetilde \pi_1 & \cong  \widetilde {\tau_{a_1}} \times \ldots \times  \widetilde{\tau_{a_{l}}} 
	  \cong  {}^{\theta_{m_{a_1}}}{\tau_{a_1}} \times \ldots \times  {}^{\theta_{m_{a_l}}}{\tau_{a_{l}}} 
	 \cong {}^{\theta_r} \pi_1 
	\end{align*}
	Moreover, we have that $\pi_1 \cong {}^{\theta_r} \widetilde \pi_1$.  
Since $\tau$ is regular, by \Cref{pairwise-inequiv}, the discrete series $\tau_i$ are pairwise inequivalent; therefore, by \cite[Theorem 9.7(b)]{zelevinsky1980}, we have that $\pi_1 \ncong \pi_3$. 
That is, we have ${}^{\theta} \widetilde{\pi_1} \cong \pi_1 \ncong \pi_3$.  
In particular, ${}^{\gamma^{-1}}(\mathcal F_\Theta^y(\tau))$ is not $H_\bullet$-distinguished, and $\mathcal F_\Theta^y(\tau)$ is not $M_\Theta^{\theta}$-distinguished.
\end{proof}

\subsection*{Acknowledgements}
The author would like to thank his doctoral advisor, Fiona Murnaghan, for her support and guidance throughout the work on this project.  Thank you to U.~K.~Anandavardhanan and Yiannis Sakellaridis for helpful comments.  
Thank you to Shaun Stevens for pointing out an error in a previous version of \Cref{prop-lin-even-dist-ds}\eqref{prop-lin-even-dist-ds--4}. Finally, thank you to the anonymous referee for many helpful suggestions that improved this article.
\bibliographystyle{amsplain}

\bibliography{jerrod-refs}
\end{document}